\documentclass{article}

\usepackage{color}
\usepackage{amsthm,amssymb,amsfonts,mathrsfs}
\usepackage{amsmath}
\newtheorem{Theorem}{Theorem}[section]

\theoremstyle{Definition}

\newtheorem{Definition}{Definition}[section]
\newtheorem{Remark}{Remark}

\newtheorem{Assumptions}{Hypothesis}[section]

\def\io{\int_\omega}
\def\R{{\mathbb{R}}}

\def\cH{\mathcal H}

\def\itau12{\int_{\tau_1}^{\tau_2}}

\def\ds{\displaystyle}

\title {Controllability for a degenerate cascade system}
\author{{\sc Idriss Boutaayamou}\thanks{This work has been done while Idriss Boutaayamou was visiting Università degli Studi di Bari, under the Young Investigator Training Program 2018 supported by ACRI (Associazione di Fondazioni e di Casse di Risparmio Spa) and Università degli Studi di Urbino Carlo Bo.}\\
D\'epartement de Math\'ematiques Informatiques et Gestion,\\
Facult\'e Polydisciplinaires de Ouarzazate,
Universit\'e Ibn Zohr,\\
B.P. 638, Ouarzazate 45000, Morocco\\ email:
dsboutaayamou@gmail.com\\
{\sc Genni Fragnelli}\thanks{The author is a member of the Gruppo Nazionale per l'Analisi Ma\-te\-matica, la Probabilit\`a e le loro Applicazioni (GNAMPA) of the
Istituto Nazionale di Alta Matematica (INdAM) and she is supported by the FFABR {\it Fondo
per il finanziamento delle attivit\`a base di ricerca} 2017, by the INdAM- GNAMPA Project 2019 {\it Controllabilit\`a di PDE in modelli fisici e in scienze della vita} and by PRIN 2017-2019 {\it Qualitative and quantitative aspects of nonlinear PDEs}.}\\
Dipartimento di Matematica\\ Universit\`{a} di Bari "Aldo Moro"\\
Via
E. Orabona 4\\ 70125 Bari - Italy\\ email: genni.fragnelli@uniba.it}

\date{}
\begin{document}

\maketitle

\vspace{0.3cm}

\centerline{ {\it  }}

\begin{abstract}
In this paper we consider a cascade system in non divergence form which models the interaction between two different species, the first one can be seen as a predator and the other as a prey. Both of them depend on time, on age and on space. Moreover, the diffusion coefficients degenerate at the boundary of domain. We study,  in particular, null controllability of the system via the observability inequality for the non homogeneous adjoint problem, which is deduced by Carleman estimates.
\end{abstract}
%%%%%%%%%%%%%%%%%%%%%%%%%%%%%%%%%%%%
%%%%%%%%%%%%%%%%%%%%%%%%%%%%%%%%%%%
%%%%%%%%%%%%%%%%%%%%%%%%%%%%%%%%%%%%
%%%%%%%%%%%%%%%%%%%%%%%%%%%%%%%%%%%%
%%%%%%%%%%%%%%%%%%%%%%%%%%%%%%%%%%%%
%%%%%%%%%%%%%%%%%%%%%%%%%%%%%%%%%%%%%%%%
%%%%%%%%%%%%%%%%%%%%%%%%%%%%%%%%%%%%%%%


\noindent 2000AMS Subject Classification: 35Q93, 93B05, 93B07, 34H15, 35A23, 35B99

\section{Introduction}
This article is devoted to study the controllability properties for the following linear degenerate coupled systems in one space dimension:
\begin{equation} \label{0}
\begin{cases}
\ds\frac{\partial y}{\partial t}+\frac{\partial y}{\partial a}- k_1(x)y_{xx} + \mu_{11}(t,x,a) y + \mu_{13}(a)y=\bar g(t,x,a) \chi_{\omega}, & (t,a,x)\in Q, 
\\
\ds\frac{\partial w}{\partial t}+\frac{\partial w}{\partial a}-k_2(x)w_{xx} + \mu_{22}(t,x,a) w-\mu (t,x,a) y + \mu_{23}(a) w=0, & (t,a,x)\in Q,
\\
y(t, a, 0) = y(t,a, 1) = w(t,a, 0) = w(t, a, 1) = 0, & (t,a)\in Q_{T,A},\\
y(0, a, x) =y_0(a,x),\quad w(0,a, x) = w_0(a,x), &  (a,x) \in Q_{A,1},\\
y(t, 0, x)=\int_0^A \gamma_1 (a, x)y (t, a, x) da,  & (t,x) \in Q_{T,1},\\
w(t, 0, x)=\int_0^A \gamma_2 (a, x)w(t, a, x) da,  & (t,x) \in Q_{T,1}.
\end{cases}
\end{equation}
Here 
$Q:=(0,T)\times (0,A)\times (0,1)$, $Q_{T,A} := (0,T)\times (0,A)$, $Q_{A,1}:=(0,A)\times(0,1)$ and
$Q_{T,1}:=(0,T)\times(0,1)$. Moreover, $y(t,a,x)$ and $w(t,a,x)$ are the distribution of certain individuals  at location $x \in (0,1)$, at time $t\in(0,T)$, where $T$ is fixed, and of age $a\in (0,A)$.
$A$ is the maximal age of life, while $\gamma_i$, $i=1,2$, are the natural fertility. Thus, the formulas $\int_0^A \gamma_1 y da$,  $\int_0^A \gamma_2 w da$ denote the distributions of newborn individuals at time $t$ and location $x$. Moreover, $\mu_{1i}$, $i=1,3$, $\mu_{2j}, j=2,3$, are the natural death rates and are such that  $\mu_{11}, \mu_{22} \in C(\bar Q, \R_+)$.The functions $\mu_{13}, \mu_{23} \in C([0,A), R_+)$ and $\int_0^A \mu_{13}(a)da=+\infty$, $\int_0^A \mu_{23}(a)da=+\infty$;  On the other hand $\mu  \in C(\bar Q, \R_+)$ can be seen as the rate of influence of the first population on the second one. Hence the term $\mu y$ can be considered as a control in the second equation; thus, $y$ is a predator and $w$ is a prey.  For more details about the modelling of such system (commonly named Lotka-McKendrick system) and the biological significance of the hypotheses, we refer to Webb \cite{web}. Finally, $ \chi_\omega$ is the characteristic function of  the control region $\omega\subset \subset(0,1)$, the control $\bar g$ belongs to a suitable Lebesgue space and the functions $k_i$, $i=1,2,$ are the dispersion coefficients and degenerate at the boundary of the state space. In particular,  we say that a function $k$ is 
\begin{Definition}\label{Wdef} {\bf Weakly degenerate (WD) at $0$ (or at $1$)}  if $k \in  W^{1,1}([0,1])$,
\[
 k>0 \text{ in }(0,1],  \;  k(0)=0 \; (\text{or } k>0 \text{ in }[0,1)  \text{ and } k(1)=0)
 \]
and there exist $M_1\in (0,1)$ such that
$x k'(x)  \le M_1k(x) $  for almost every $x\in [0,1]$ (or there exist $M_2\in (0,1)$ such that
$(x-1)k'(x) \le M_2 k(x)$ for almost every  $x\in [0,1]$).
 \end{Definition} 
or
\begin{Definition}\label{Sdef} {\bf Strongly degenerate (SD)  at $0$ (or at $1$) }  if $k \in  W^{1,\infty}([0,1])$,
\[
 k>0 \text{ in }(0,1), \;  k(0)=0 \;  (\text{or } k>0 \text{ in }[0,1)  \text{ and } k(1)=0)
 \]
and there exist $M_1\in [1,2)$ such that
$x k'(x)  \le M_1k(x) $ for almost every $x\in [0,1]$ (or there exist $ M_2\in [1,2)$ such that
$(x-1)k'(x) \le M_2 k(x)$ for almost every $x\in [0,1]$).
 \end{Definition} 
 For example, as $k$ one can consider $k(x)=x^\alpha$ or $k(x)=(1-x)^\beta,$ with $\alpha, \beta >0$.

Observe that
thanks to the following transformations 
\[u(t,a, x):= e^{\int_0^a \mu_{13}(\gamma)d\gamma}y(t,a,x) \quad
\text{ and } \quad
v(t,a, x):= e^{\int_0^a \mu_{23}(\gamma)d\gamma}w(t,a,x),\] 
\eqref{0} can be rewritten as

\begin{equation} \label{1}
\begin{cases}
\ds\frac{\partial u}{\partial t}+\frac{\partial u}{\partial a}- k_1(x)u_{xx} + \mu_{11}(t,x,a) u =g(t,x,a) \chi_{\omega}, & (t,a,x)\in Q, 
\\
\ds\frac{\partial v}{\partial t}+\frac{\partial v}{\partial a}-k_2(x)v_{xx} + \mu_{22}(t,x,a) v-\mu_{21}(t,x,a) u=0, & (t,a,x)\in Q,
\\
u(t, a, 0) = u(t,a, 1) = v(t,a, 0) = v(t, a, 1) = 0, & (t,a)\in Q_{T,A},\\
u(0, a, x) =u_0(a,x),\quad v(0,a, x) = v_0(a,x), &  (a,x) \in Q_{A,1},\\
u(t, 0, x)=\int_0^A \beta_1 (a, x)u (t, a, x) da,  & (t,x) \in Q_{T,1},\\
v(t, 0, x)=\int_0^A \beta_2 (a, x)v(t, a, x) da,  & (t,x) \in Q_{T,1}.
\end{cases}
\end{equation}
where $u_0(a, x):=  e^{\int_0^a \mu_{13}(\gamma)d\tau}y_0(a,x)$, $v_0(a, x):=  e^{\int_0^a \mu_{23}(\gamma)d\tau}w_0(a,x)$, $g(t,a,x):= e^{\int_0^a \mu_{13}(\tau)d\tau}\bar g(t,a,x)$,  $\mu_{21}(t,a,x):=e^{\int_0^a(\mu_{23}(\gamma) -\mu_{13}(\gamma) )d\gamma}\mu(t,a,x) $ and $\beta_i(a, x):=e^{-\int_0^a \mu_{i3}(\tau)d\tau} \gamma_i (a,x)  $, $i=1,2$. Observe that the two integrals $e^{\int_0^a \mu_{13}(\gamma)d\gamma}$ and $e^{\int_0^a \mu_{23}(\gamma)d\gamma}$  are the inverse of the probability of survival of an individual from age 0 to a. Thus, in place of \eqref{0}, it is not restrictive to consider \eqref{1}, as we will do from now on in the rest of the paper. 

Population dynamics models are the subject of numerous papers where they were investigated from many points of view. One among the notable questions in these models was the controllability issue for age and space structured population dynamics models which were studied in an intensive literature. 
In this context, we can cite the pioneering papers of V. Barbu and al. in \cite{iannelli}, B. Ainseba and S. Anita in \cite{Ain4, Ain3, Ain1, Ain2}, and recently \cite{Maity1,Maity2}.

In \cite{iannelli}, the authors proved the null controllability for a population dynamics model without diffusion, both in the cases of migration and birth control for $T\geq A$, showing directly an appropriate observability inequality for the associated adjoint system. Moreover, they concluded  that, in the case of the migration control, only a classes of age was controlled in contrast with the birth control, which allows to steer all population to extinction.

 In \cite{Ain4, Ain3, Ain1, Ain2}, the null controllability result for the age-space structured model  was established  when the diffusion coefficient $k\equiv1$ and for any space dimension, exploiting the results obtained for heat equation in \cite{Fursikov}. 
 
In \cite{Maity2}  the authors studied the null controllability of a linear system coming from a population dynamics model with age structuring and spatial diffusion (of Lotka-McKendrick type). They considered a control that is localized in the space variable as well as with respect to the age . In their work the first novelty was that the age interval, in which the control needs to be active, can be arbitrarily small  without needing to exclude a neighborhood of zero. The second one is that the whole population can be steered into zero in a uniform time, without excluding some interval of low ages.

In \cite{Maity1} the authors considered a linear infinite dimensional system obtained by grafting an age structure. Such systems appear essentially in population dynamics with age structured when phenomena like spatial diffusion or transport are also taken into consideration. They asserted that if the initial system is null controllable in a time small enough, then the structured system is also null controllable in a time depending on the various involved parameters. 

In \cite{ech}, B. Ainseba and al. studied a more general case allowing the dispersion coefficient, $k$, to depend on the variable $x$ and to verify $k(0)=0$. Moreover, they assume the condition $T\geq A$ (as in \cite{iannelli}) and this constitutes a restrictiveness on the "optimality" of the control time $T$ since it means, for example, that for a pest population, whose the maximal age $A$ may equal  many days (maybe many months or years), we need much time to bring the population to the zero equilibrium. In the same trend and to overcome the condition $T\geq A$, in \cite{bout_ech,man} the authors used the fixed point technique (in particular the  Leray-Schauder Theorem) to obtain the null controllability for an intermediate system.
In \cite{f_JMPA, Genni2, Genni3}, the author considered the divergence and non divergence form and used only Carleman estimates and a technique based on cut-off functions, making the proof slimmer and easier to read. However, observe that  all the previous papers deal with a single equation.

Up to now, little is known about the null controllability for a population dynamics  cascade system both in degenerate and nondegenerate cases. Among the  papers treating this argument we recall \cite{anita2005, maniar11, idriss2016, tereza,  zhao2005} and the references therein. In particular, in \cite{anita2005} the authors studied a coupled
reaction-diffusion equation describing interactions between a prey population and a predator population. The goal of the work was to look for a suitable control supported on
a small spatial subdomain which guarantees the stabilization of the predator population
to zero. In \cite{zhao2005}, the objective was different. More precisely, the authors considered an
age-dependent prey-predator system and they proved the existence and uniqueness for
an optimal control which gives the maximal harvest via the
study of the optimal harvesting problem associated to their coupled model.
However, the previous results were found in the case when the diffusion coeffcients are
constants. This leads to generalize the model of \cite{anita2005}  in \cite{maniar11}, in \cite{tereza} and  in \cite{idriss2016}, where a linear or a semilinear parabolic cascade system with two different diffusion coefficients is considered. In particular, they can degenerate at the boundary or in the interior of the space domain. Finally, in \cite{salhi} the author considered a system similar to the previous ones but with the addition of a singular term. However, in all the previous papers the prey and the predator depend only on  time and on  space. To our knowledge, this is the first paper where the two populations depend also on  age and the diffusion coefficients degenerate at the boundary of the domain. 

The paper is organized as follows: in Section \ref{section2} we introduce the suitable Hilbert spaces where the problem is well posed and we give, using the semigroup approach, the existence theorem; in Section \ref{section3} we prove Carleman estimates and observability inequalities for the non homogeneous associated adjoint problem via the Carleman estimates given in \cite{f_JMPA} for a single equation. Finally, in section \ref{section4} we deduce the null controllability result for an intermediate system and hence for the initial problem \eqref{1}.

\section{Assumptions and well-posedness} \label{section2}

Through the paper, we assume that the rates $\mu_{ij}$ and $\beta_i$,  $i,j=1,2$ satisfy the following:
\begin{Assumptions}\label{ratesAss}
 The  functions $\mu_{11}$, $\mu_{2i}$ and $\beta_i$,  $i=1,2$, are such that
\begin{equation}\label{3}
\begin{aligned}
&\bullet \beta_i \in C(\bar Q_{A,1}) \text{ and } \beta_i \geq0  \text{ in } Q_{A,1}, \\
&\bullet \mu_{11}, \mu_{2i} \in C(\bar Q) \text{ and } \mu_{11},  \mu_{2i}\geq0\text{ in } Q.
\end{aligned}
\end{equation}
\end{Assumptions}

To study the well posedness of \eqref{1}, 
we consider four situations, namely the weakly-weakly degenerate case (WWD), i.e. the case when $k_1$ and $k_2$ are both (WD), the strongly-strongly degenerate case (SSD), i.e. the case when $k_1$ and $k_2$ are both (SD), the weakly-strongly degenerate case (WSD), i.e the case when $k_1$  is (WD) and $k_2$ is (SD), and the 
strongly-weakly degenerate case (SWD),  i.e the case when $k_1$ is (SD) and $k_2$ is (WD). Towards this end, as in  \cite{cfr},\cite{f_JMPA} or in \cite{fm1} we consider
the following  weighted spaces:
\[
L^2_{\frac{1}{k_i}}(0,1):= \left\{ u \in L^2(0,1): \int_0^1 \frac{u^2}{k_i}dx <\infty\right\},
\]
\[
 \mathcal{H}_{\frac{1}{k_i}}^1(0, 1)
:=L^2_{\frac{1}{k_i}}(0,1)\cap H^1_0(0,1)
\]
and
\[
\mathcal{H}_{\frac{1}{k_i}}^2(0, 1):=\big\{ u \in  \mathcal{H}_{\frac{1}{k_i}}^1(0, 1):  u_x
\in H^1(0, 1)\big\},
\]
$i=1,2$,
with the norms
\[ \|u\|_{L^2_{\frac{1}{k_i}}(0,1)}^2:= \int_0^1
\frac{u^2}{k_i} dx, \quad \forall\, u\in L^2_{\frac{1}{k_i}}(0,1),\]
\[
\|u\|^2_{\cH^1_{\frac{1}{k_i}}}:=\|u\|_{L^2_{\frac{1}{k_i}}(0,1)}^2 +
\|u_x\|^2_{L^2(0,1)}, \quad \forall\, u\in
\cH^1_{\frac{1}{k_i}}(0,1)\]
 and
\[
\|u\|_{\cH^2_{\frac{1}{k_i}}(0,1)}^2
:=\|u\|_{\cH^1_{\frac{1}{k_i}}(0,1)}^2 +
\|k_iu_{xx}\|^2_{L^2_{\frac{1}{k_i}}(0,1)},\quad \forall\,u\in
\cH^2_{\frac{1}{k_i}}(0,1).
\]
Indeed, it is a trivial fact that, if $u_x\in H^1(0,1)$, then $k_iu_{xx}
\in L^2_{\frac{1}{k_i}}(0,1)$, so that the norm for
$\cH^2_{\frac{1}{k_i}}(0,1)$ is well defined. For this reason in \cite{fm}, $\cH^2_{\frac{1}{k_i}}(0,1)$ is written in a more appealing way as
\[
\cH^2_{\frac{1}{k_i}}(0,1) :=\Big\{ u \in \cH^1_{\frac{1}{k_i}}(0,1) \,
\big| \, u_x\in H^1(0,1) \mbox{ and } k_iu_{xx} \in
L^2_{\frac{1}{k_i}}(0,1)\Big\}.
\]

As in \cite[Theorem 2.3]{cfr} or in \cite[Theorem 2.2]{fm}, we can prove that, for $i=1,2$, the  operators
$(A_i,D(A_i))$ defined by $A_iu := k_iu_{xx}$, with
$ u \in D(A_i) =\cH^2_{\frac{1}{k_i}}(0,1) $,
are closed, self-adjoint, negative  with dense domain in $L^2_{\frac{1}{k_i}}(0, 1)$. Moreover, setting $ \mathcal A_a u := \ds \frac{\partial  u}{\partial a}$, we have that the operators
%\[
%\mathcal A_a u := \frac{\partial  u}{\partial a}
%\]
%for $u \in D(\mathcal A_a)= H^1(0,A).$ Again, we have that $(\mathcal A_a, D(\mathcal A_a))$ generates a strongly continuous semigroup on $L^2(0,A)$. Hence
\[
\mathcal A_iu:= \mathcal A_a u - 
A_i u,
\]
%for $u \in L^2(0,A; D(\mathcal A_0)) \cap H^1(0,A; L^2(0,1))$ generates a strongly continuous semigroup on $L^2(Q_{A,1})$.
%By \cite[Example III.5.9]{en} and \cite[Theorem 3.4.4]{lm}, one can show
%the existence of a unique solution for the model \eqref{1}. In particular, 
%setting
%\[
%\mathcal A(t)u:= \frac{\partial u}{\partial a}
%+\mu(t, a, x)u-\mathcal A_0 u,
%\]
for 
\[
u \in D(\mathcal A_i) =\left\{u \in L^2(0,A;D(A_i)) : \frac{\partial u}{\partial a} \in  L^2(0,A;H^1_{\frac{1}{k_i}}(0,1)), u(0, x)= \int_0^A \beta_i(a, x) u(a, x) da\right\},
\]
generate two strongly continuous semigroups on $L^2(0,A) \times L^2 _{{\frac{1}{k_i}}}(0,1)$ (see also \cite{iannelli}).  Now,
 consider the Hilbert spaces $\mathbb{L}=\mathrm{L}^{2}_{\frac{1}{k_1}}(0,1)\times \mathrm{L}^{2}_{\frac{1}{k_2}}(0,1)$, $\mathbb{H}= \mathcal H^1_{\frac{1}{k_1}}(0,1)\times \mathcal H^1_{\frac{1}{k_2}}(0,1)$ , $\mathbb{K}= \mathcal H^{2}_{\frac{1}{k_1}}(0,1)\times \mathcal  H^{2}_{\frac{1}{k_2}}(0,1)$ 
and define
\[
\mathcal U(t)=\begin{pmatrix}
u(t)\\
v(t)
\end{pmatrix}, \quad \mathcal{A}=\begin{pmatrix}
\mathcal A_1&0\\
0&\mathcal A_2
\end{pmatrix}, \quad D(\mathcal{A})=D(\mathcal A_1)\times D(\mathcal A_2),
\]
\[
B(t)=\begin{pmatrix}
M_{\mu_{11}(t)}&0\\
-M_{\mu_{21}(t)}&M_{\mu_{22}(t)}
\end{pmatrix} \; \text{ and }\;
f(t,\cdot,\cdot)=\begin{pmatrix}
g(t,\cdot, \cdot) \chi_\omega \\
0
\end{pmatrix},
\]
where $M_{\mu_{2i}(t)}y= \mu_{2i}(t)y$, $i=1,2$. Analogously for $M_{\mu_{11}}$. Observe that  the operator $\mathcal{A}$ is a diagonal matrix and $B$ is a triangular one. Moreove, $B$ can be seen as a bounded perturbation of $(\mathcal A, D(\mathcal A))$. Clealy, \eqref{1} can be rewritten as
\begin{equation}\label{1new}
\begin{cases}
\mathcal U'(t) +\mathcal{A}\mathcal U(t) +B(t)\mathcal U(t)=f(t),\\
\mathcal U(0)=\begin{pmatrix}
u_0\\
v_0
\end{pmatrix}.
\end{cases}
\end{equation}

Setting $L^2_{{\frac{1}{k_1}}}(Q):= L^2(Q_{T,A}; L^2_{\frac{1}{k_1}}(0,1))$ and $L^2_{{\frac{1}{k_i}}}(Q_{A,1}):= L^2(0,A;L^2_{\frac{1}{k_i}}(0,1))$, $i=1,2$, the following well posedness and regularity results hold (see \cite{maniar11} and the references therein).

\begin{Theorem}\label{esistenza}Assume Hypothesis \ref{ratesAss} and suppose that $k_i$, $i=1,2$, are $(SD)$ or $(WD)$ at $0$ and/or at $1$.
The operator $\mathcal{A}$ generates a strongly
continuous semigroup $(\mathcal T(t))_{t\geq0}$.
Moreover, for all $g \in  L^2_{\frac{1}{k_1}}(Q)$ and for all $(u_0, v_0) \in L^2(0,A;\mathbb{L})$, system \eqref{1} admits a unique solution
\[(u,v) \in C([0,T]; L^2(0,A; 
\mathbb{L})) \cap L^2 (0,T; H^1(0,A;\mathbb{H}))\]
 of \eqref{1}. In addition, if $g \equiv 0$, 
$
(u, v) \in C^1\big([0,T]; L^2(0,A; 
\mathbb{L})\big).
$ 
\end{Theorem}
\section{Carleman estimates for the adjoint cascade system} \label{section3}
In general, null controllability
for a linear parabolic system
is, roughly speaking, equivalent to or follows by the observability for the associated  adjoint problem
\begin{equation}\label{adjoint}
\begin{cases}
\ds \frac{\partial z}{\partial t} + \frac{\partial z}{\partial a}
+k_1(x)z_{xx}-\mu_{11}(t, a, x)z+\mu_{21}(t, a, x)y =- \beta_1(a,x)z(t,0,x) ,& (t,a,x) \in Q,  \\
\ds \frac{\partial y}{\partial t} + \frac{\partial y}{\partial a}
+k_2(x)y_{xx}-\mu_{22}(t, a, x)y =-\beta_2(a,x)y(t,0,x),& (t,a,x) \in Q,
\\
  z(t, a, 0)=z(t, a, 1)= y(t, a, 0)=y(t, a, 1)=0, & (t,a) \in Q_{T,A},\\
  z(t,A,x)=y(t,A,x)=0, & (t,x) \in Q_{T,1},\\
z(T,a,x)= z_T(a,x) \in L^2_{{\frac{1}{k_1}}}(Q_{A,1}), & (a,x) \in Q_{A,1},\\
y(T,a,x)= y_T(a,x) \in L^2_{{\frac{1}{k_2}}}(Q_{A,1}), & (a,x) \in Q_{A,1},
\end{cases}
\end{equation}
where $z_T(A,x)=0=y_T(A,x)$ for all $x \in [0,1]$. 
 Thus, the crucial point is to prove such an  inequality via, for example,
 Carleman estimates for \eqref{adjoint}. This is the goal of this section.

\vspace{0.5cm}

\subsection{Preliminary results for Carleman estimates}
In this subsection we will recall some results of \cite{f_JMPA} for a single equation, that  will be crucial to prove Carleman estimates for
the non homogeneous adjoint problem of \eqref{1}.
To this aim, we consider the system
\begin{equation}\label{single}
\begin{cases}
\ds \frac{\partial v}{\partial t} + \frac{\partial v}{\partial a}
+k(x)v_{xx}-\mu(t, a, x)v =f ,& (t,a,x) \in Q,\\
  v(t, a, 0)=v(t, a, 1)=0, & (t,a) \in Q_{T,A},\\
   v(t,A,x)=0, & (t,x) \in Q_{T,1}.
\end{cases}
\end{equation}
If the function $k$ degenerates at $0$ we make the following assumptions:
\begin{Assumptions}\label{BAss01}
\begin{enumerate}
\item
 The function
$k\in C^0[0,1]\bigcap C^2(0,1]$  is such that $k(0)=0$, $k>0$ on
$(0,1]$ and there exist $\varepsilon\in (0,1]$ and $M\in (0,2)$ such that
the function $\displaystyle\frac{xk_x}{k(x)}$
$\in \:L^{\infty}(0,\varepsilon)$,
\small$\displaystyle\frac {xk_x(x)}{k(x)} \le M $, for all $x \in (0,1]$, and
$\displaystyle \left( \frac{xk_x(x)}{k(x)}\right)_{x} \in L^{\infty}(0,\varepsilon)$.
\item $\mu \in C(\bar Q)$ and $\mu \ge 0$ in $Q$.
\end{enumerate} 
\end{Assumptions}
\begin{Assumptions}\label{Assw}
The control set $\omega$ is such that
\begin{equation}\label{omega}
\omega=  (\alpha, \beta)  \subset\subset  (0,1).
\end{equation}
\end{Assumptions}
\noindent Now, let us introduce the weight function
\begin{equation}\label{13}
\varphi(t,a,x):=\Theta(t,a)(p(x) - 2 \|p\|_{L^\infty(0,1)}),
\end{equation} 
where 
\begin{equation}\label{571}
\begin{gathered}
\Theta(t, a)= \frac{1}{t^{4}(T-t)^{4}a^{4}}
\end{gathered}
\end{equation}
and
$\displaystyle p(x):=\int_0^x\frac{y}{k(y)}~e^{Ry^2}dy$, with $R>0$.

Then, defining $ \mathcal H^1_{{\frac{1}{k}}}(0,1)$ and $ \mathcal H^2_{{\frac{1}{k}}}(0,1)$ as $ \mathcal H^1_{{\frac{1}{k_i}}}(0,1)$ and $ \mathcal H^2_{{\frac{1}{k_i}}}(0,1)$, respectively, one has:

\begin{Theorem}\label{Cor2}[see \cite[Theorem 4.1]{f_JMPA}] Assume Hypotheses $\ref{BAss01}$ and $\ref{Assw}$. Then,
there exist two  strictly positive constants $C$ and $s_0$ such that every
solution $v \in L^2\big(Q_{T,A}; \mathcal H^2_{{\frac{1}{k}}}(0,1)\big) \cap H^1\big(0, T; H^1(0,A;\mathcal H^1_{{\frac{1}{k}}}(0,1))\big)$
of 
\eqref{single}
satisfies, for all $s \ge s_0$,
\[
\begin{aligned}
\int_{Q}\left(s \Theta v_x^2
                + s^3\Theta^3\text{\small$\displaystyle\Big(\frac{x}{k}\Big)^2$\normalsize}
                  v^2\right)e^{2s\varphi}dxdadt
&\le
C\left(\int_{Q}\text{\small$\frac{f^{2}}{k}$\normalsize}~dxdadt+ \int_{Q_{T,A}}\int_ \omega \frac{v^2}{k} dx dadt\right).
\end{aligned}\]
\end{Theorem}
On the other hand, if $k$ degenerates at $1$, we assume in place of Hypothesis \ref{BAss01}, the following one:

\begin{Assumptions}\label{BAss02} 
\begin{enumerate}
\item
The function
$k\in C^0[0,1]\bigcap C^2(0,1]$ is such that $k(1)=0$, $k>0$ on
$(0,1]$ and  there exist $\varepsilon_i\in (0,1]$ and $M\in (0,2)$ such that
the function $\displaystyle\frac{(x-1)k_{x}}{k(x)}$\normalsize\:
$\in L^{\infty}(1-\varepsilon,1)$,
$\displaystyle\frac {(x-1)k_{x}(x)}{k(x)} \le M$, for all $x \in [0,1)$,
and $\displaystyle\left( \frac{(x-1)k_{x}(x)}{k(x)}\right)_{x}\in L^{\infty}(1-\varepsilon,1)$.
\item $\mu \in C(\bar Q)$ and $\mu \ge 0$ in $Q$.
\end{enumerate}
\end{Assumptions}
Then define:
\begin{equation}\label{13new}
\bar\varphi(t,a,x):=\Theta(t,a)(\bar p(x) - 2 \|\bar p\|_{L^\infty(0,1)}),
\end{equation} 
where $\Theta$ is as before and
$\displaystyle \bar p(x):=\int_0^x\frac{y-1}{k(y)}~e^{R(y-1)^2}dy$, with $R>0$.
Hence, the analogous result of Theorem \ref{Cor2} holds:

\begin{Theorem}\label{Cor1}[see \cite[Theorem 4.2]{f_JMPA}] Assume Hypotheses $\ref{Assw}$ and $\ref{BAss02}$. Then,
there exist two  strictly positive constants $C$ and $s_0$ such that every
solution $v \in L^2\big(Q_{T,A}; \mathcal H^2_{{\frac{1}{k}}}(0,1)\big) \cap H^1\big(0, T; H^1(0,A;\mathcal H^1_{{\frac{1}{k}}}(0,1))\big)$ of 
\eqref{single}
satisfies, for all $s \ge s_0$,
\[
\begin{aligned}
\int_{Q}\left(s \Theta v_x^2
                + s^3\Theta^3\text{\small$\displaystyle\Big(\frac{x-1}{k}\Big)^2$\normalsize}
                  v^2\right)e^{2s\varphi}dxdadt
&\le
C\left(\int_{Q}\text{\small$\frac{f^{2}}{k}$\normalsize}~dxdadt+ \int_0^T \int_0^A\int_ \omega \frac{v^2}{k} dx dadt\right).
\end{aligned}\]
\end{Theorem}

\vspace{1cm}
Now, we are ready to prove Carleman estimates for the non homogeneous adjoint problem of \eqref{1}:
\begin{equation}\label{adjoint_f}
\begin{cases}
\ds \frac{\partial z}{\partial t} + \frac{\partial z}{\partial a}
+k_1(x)z_{xx}-\mu_{11}(t, a, x)z+\mu_{21}(t, a, x)y =f ,& (t,a,x) \in Q,  \\
\ds \frac{\partial y}{\partial t} + \frac{\partial y}{\partial a}
+k_2(x)y_{xx}-\mu_{22}(t, a, x)y = h,& (t,a,x) \in Q
\\
  z(t, a, 0)=z(t, a, 1)= y(t, a, 0)=y(t, a, 1)=0, & (t,a) \in Q_{T,A},\\
  z(t,A,x)=y(t,A,x)=0, & (t,x) \in Q_{T,1},
\end{cases}
\end{equation}
where $f \in L^2_{\frac{1}{k_1}}(Q)$ and $h\in L^2_{\frac{1}{k_2}}(Q)$.
 To this aim, we distinguish between the case when $k_i(0)=0$ and $k_i(1)=0$, $i=1,2$.
\subsection{Carleman inequalities and observability inequalities when the degeneracy is at $0$.}\label{sec-3-1}

An immediate consequence of Theorem \ref{Cor2} is  the next $\omega-$local Carleman estimate for \eqref{adjoint_f}. Assume
\begin{Assumptions}\label{AssP} The functions
$k_i\in C^0[0,1]\bigcap C^2(0,1]$, $i=1,2$, 
 are such that $k_i(0)=0$, $k_i>0$ on
$(0,1]$ and 
\begin{enumerate}
\item $k_i$ satisfies Hypothesis  \ref{BAss01}.1, for $i=1,2$,
%\item there exist $\varepsilon_i\in (0,1]$ and $M_i\in (0,2)$ such that
%the functions $\displaystyle\frac{xk_{i,x}}{k_i(x)}$
%$\in \:L^{\infty}(0,\varepsilon_i)$,
%\small$\displaystyle\frac {xk_{i,x}(x)}{k_{i}(x)} \le M_i$, for all $x \in (0,1]$, and
%$\displaystyle \left( \frac{xk_{i,x}(x)}{k_{i}(x)}\right)_{x} \in L^{\infty}(0,\varepsilon_i)$, $i=1,2$.
\item there exists a positive constant $C$ such that $k_1(x)\ge k_2(x)$ for all $x \in [0,1]$.
\end{enumerate}
\end{Assumptions}
\begin{Theorem}\label{Carleman}
Assume Hypotheses $\ref{ratesAss}$, $\ref{Assw}$  and $\ref{AssP}$. Take  $f \in L^2_{\frac{1}{k_1}}(Q)$ and $h\in L^2_{\frac{1}{k_2}}(Q)$. Then,
there exist two strictly positive constants $C$ and $s_0$ such that every
solution $(z,y)\in L^2\big(Q_{T,A}; \mathbb{K}\big) \cap H^1\big(0, T; H^1(0,A;\mathbb{H})\big)$ of \eqref{adjoint_f}
satisfies, for all $s \ge s_0$,
\[
\begin{aligned}
&\int_{Q}\left(s \Theta z_x^2
                + s^3\Theta^3\text{\small$\displaystyle\Big(\frac{x}{k_1}\Big)^2$\normalsize}
                  z^2\right)e^{2s\varphi_1}dxdadt  
 \le  C \int_{Q}(f^2+ y^2)\frac{1}{k_1}dxdadt +C \int_{Q_{T,A}}\int_ \omega \frac{z^2}{k_1} dx dadt
\end{aligned}\]
and
\[
\begin{aligned}
\int_{Q}\left(s \Theta y_x^2
                + s^3\Theta^3\Big(\frac{x}{k_2}\Big)^2
                  y^2\right)e^{2s\varphi_2}dxdadt
&\le
C\left(\int_{Q}\text{\small$\frac{h^{2}}{k_2}$\normalsize}~dxdadt+ \int_{Q_{T,A}}\int_ \omega \frac{y^2}{k_2} dx dadt\right).
\end{aligned}\]
Here \begin{equation}\label{funzioni}
\varphi_i(t,a,x):=\Theta(t,a)(p_i(x) - 2 \|p_i\|_{L^\infty(0,1)}),
\end{equation}
where $\Theta$ is as in \eqref{571} and
$\displaystyle p_i(x):=\int_0^x\frac{y}{k_i(y)}~e^{Ry^2}dy$, with $R>0$, and $i=1,2$.
\end{Theorem}
\begin{proof} Observe that \eqref{adjoint_f}
can be rewritten as
\begin{equation}\label{adjoint_1}
\begin{cases}
\ds \frac{\partial z}{\partial t} + \frac{\partial z}{\partial a}
+k_1(x)z_{xx}-\mu_{11}(t, a, x)z =f-\mu_{21}(t, a, x)y=:F ,& (t,a,x) \in Q,  \\
  z(t, a, 0)=z(t, a, 1)=0, & (t,a) \in Q_{T,A},\\
  z(t,A,x)=0, & (t,x) \in Q_{T,1},
\end{cases}
\end{equation}
where $y$ satisfies
\begin{equation}\label{adjoint_2}
\begin{cases}
\ds \frac{\partial y}{\partial t} + \frac{\partial y}{\partial a}
+k_2(x)y_{xx}-\mu_{22}(t, a, x)y =h,& (t,a,x) \in Q,
\\
 y(t, a, 0)=y(t, a, 1)=0, & (t,a) \in Q_{T,A},\\
y(t,A,x)=0, & (t,x) \in Q_{T,1}.
\end{cases}
\end{equation}
Hence, the inequality for $y$ follows immediately by Theorem \ref{Cor2} applied to \eqref{adjoint_2}. On the other hand, the estimante on $z$ follows by
Theorem \ref{Cor2} applied to \eqref{adjoint_1}. Indeed, we have
\[
\begin{aligned}
\int_{Q}\left(s \Theta z_x^2
                + s^3\Theta^3\Big(\frac{x}{k_1}\Big)^2
                  z^2\right)e^{2s\varphi_1}dxdadt
&\le
C\left(\int_{Q}\frac{F^{2}}{k_1} dxdadt+ \int_{Q_{T,A}} \int_ \omega \frac{z^2}{k_1} dx dadt\right)\\
& \le C \left(\int_{Q}\left(\frac{f^{2}}{k_1} + \frac{y^{2}}{k_1}\right) dxdadt+ \int_{Q_{T,A}}\int_ \omega \frac{z^2}{k_1} dx dadt\right).
\end{aligned}\]
Thus, the thesis follows.

\end{proof}
\begin{Remark}\label{remarkultimo}
Observe that the results of Theorems  \ref{Cor2} , \ref{Cor1} and \ref{Carleman} still hold true if we substitute the domain $(0,T)\times (0,A)$ with a general domain $(T_\gamma,T_\beta)\times (\gamma,A)$ where the required assumptions are satisfied. In this case, in place of the function $\Theta$ defined in \eqref{571}, we have to consider the weight function
\begin{equation}\label{thetatilde}
\tilde \Theta(t,a):= \frac{1}{(t-T_\gamma)^4 (T_\beta-t)^4(a-\gamma)^4}.
\end{equation}
\end{Remark}

Using the $\omega-$local Carleman estimates given in Theorem \ref{Carleman}, one can prove the next observability inequality for the adjoint system \eqref{adjoint}.
From now on, we will make an additional assumption on the birth  rates $\beta_i$, $i=1,2$:
\begin{Assumptions}\label{conditionbeta} Assume that there exist
$\bar a_i  < \Gamma:= \min \{T,A\}$
 such that
\begin{equation}\label{conditionbeta1}
\beta_i(a, x)=0 \;  \text{for all $(a, x) \in [0, \bar a_i]\times [0,1]$},  \quad i=1,2.
\end{equation}
\end{Assumptions} 
Observe that Hypothesis \ref{conditionbeta} is realistic, since $\bar a_i$, $i=1,2$, are the minimal age in which the female of the population become fertile, thus it is natural that before $\bar a_i$ there are no newborns.
\begin{Theorem}\label{Theorem4.4} Assume Hypotheses $\ref{ratesAss},$ $\ref{Assw}$, $\ref{AssP}$ and  \ref{conditionbeta}. Then, for every $\delta \in (0,A)$, 
there exists a strictly positive constant $C=C(\delta)$ such that every
solution $(z,y)\in L^2\big(Q_{T,A}; \mathbb{K}\big) \cap H^1\big(0, T; H^1(0,A;\mathbb{H})\big)$ of \eqref{adjoint}
satisfies
\begin{equation}\label{stima1}
\begin{aligned}
&\int_{Q_{A,1}} \frac{1}{k_2} y^2(T-\bar a_2, a,x) dxda \le C \left(\int_0^{\bar a_2} \int_0^1\frac{y_T^2(a,x) }{k_2}dxda + \int_{Q_{T,A}}\!\!\int_ \omega\frac{1}{k_2} y^2  dx dadt+ \int_0^T \int_0^\delta\!\!\int_ 0^1 \frac{1}{k_2} y^2 dx dadt\right)\end{aligned}\end{equation}
and
\begin{equation}\label{stima2}
\begin{aligned}
\int_{Q_{A,1}} \frac{z^2}{k_1} (T-\bar a_1,a,x) dxda& \le C\left(  \int_{Q_{T,A}}\!\!\int_ \omega\left( \frac{1}{k_1} z^2 + \frac{1}{k_2} y^2\right)dx dadt+ \int_0^T \int_0^\delta\!\!\int_ 0^1\left( \frac{1}{k_1} z^2 + \frac{1}{k_2} y^2\right)dx dadt\right)\\
&+ C\left(\int_0^{\bar a_1} \int_0^1\frac{z_T^2(a,x) }{k_1}dxda +\int_0^{\bar a_2} \int_0^1\frac{y_T^2(a,x) }{k_2}dxda\right)
\end{aligned}\end{equation}
Moreover, if $y_T(a,x)=0$ for all $(a,x) \in (0,  \bar a_2) \times (0,1)$ and  $z_T(a,x)=0$ for all $(a,x) \in (0,  \bar a_1) \times (0,1)$, one has
\[
\begin{aligned}
&\int_{Q_{A,1}} \frac{1}{k_2} y^2(T-\bar a_2, a,x) dxda \le C\int_{Q_{T,A}}\!\!\int_ \omega\frac{1}{k_2} y^2  dx dadt+ C\int_0^T \int_0^\delta\!\!\int_ 0^1 \frac{1}{k_2} y^2 dx dadt
\end{aligned}
\]
and
\[\begin{aligned}
 \int_{Q_{A,1}} \frac{z^2}{k_1} (T-\bar a_1,a,x) dxda &\le C\left( \int_{Q_{T,A}}\!\!\int_ \omega\left( \frac{1}{k_1} z^2 + \frac{1}{k_2} y^2\right)dx dadt+ \int_0^T \int_0^\delta\!\!\int_ 0^1\left( \frac{1}{k_1} z^2 + \frac{1}{k_2} y^2\right)dx dadt\right).
\end{aligned}\]
\end{Theorem}
\begin{proof}
As before, \eqref{adjoint}  can be rewritten as
\begin{equation}\label{adjoint_1'}
\begin{cases}
\ds \frac{\partial z}{\partial t} + \frac{\partial z}{\partial a}
+k_1(x)z_{xx}-\mu_{11}(t, a, x)z =-\mu_{21}(t, a, x)y - \beta_1(a,x)z(t,0,x) =:f ,& (t,a,x) \in Q,  \\
  z(t, a, 0)=z(t, a, 1)=0, & (t,a) \in Q_{T,A},\\
  z(t,A,x)=0, & (t,x) \in Q_{T,1},\\
z(T,a,x)= z_T(a,x), & (a,x) \in Q_{A,1},
\end{cases}
\end{equation}
where $y$ satisfies
\begin{equation}\label{adjoint_2'}
\begin{cases}
\ds \frac{\partial y}{\partial t} + \frac{\partial y}{\partial a}
+k_2(x)y_{xx}-\mu_{22}(t, a, x)y =-\beta_2(a,x)y(t,0,x)=:h,& (t,a,x) \in Q,
\\
 y(t, a, 0)=y(t, a, 1)=0, & (t,a) \in Q_{T,A},\\
y(t,A,x)=0, & (t,x) \in Q_{T,1},\\
y(T,a,x)= y_T(a,x), & (a,x) \in Q_{A,1}.
\end{cases}
\end{equation}
Setting,
\begin{equation}\label{tildeT}
T_i:= T-\bar a_i, \quad i=1,2.
\end{equation}
using the method of characteristic lines, the assumption on $\beta_i$, $i=1,2$, and the fact that $z(t,A,x)=y(t,A,x)=0$ for all $(t,x) \in Q_{T,1}$, one has, 
as in \cite[Theorem 4.4]{f_JMPA}, the following implicit formulas for
 every solution $z$ and $y$ of \eqref{adjoint_1'} and \eqref{adjoint_2'}, respectively:
\begin{equation}\label{implicitformula2}
z(t,a, \cdot)= S_1(T-t) z_T(T+a-t, \cdot)- \int_a^{T+a-t}S_1(s-a)G(s,t,a,\cdot)ds,
\end{equation}
where $G(s,t,a,x): =-\mu_{21}(s+t-a, s,x) y(s+t-a, s,x)$, 
if $t \ge   T_1+ a$ (observe that in this case $T+a-t \le \bar a_1$) and, setting $\Gamma:= T-A$,
\begin{equation}\label{implicitformula3}
z(t,a, \cdot)\!=\!\begin{cases}
S_1(T-t) z_T(T+a-t, \cdot)\!+\!\int_a^{T+a-t}\!S_1(s-a)\!\left(\beta_1(s, \cdot)z(s+t-a, 0, \cdot)- G(s,t,a,\cdot)\right) ds, & \Gamma +a\le t  \\
\int_a^AS_1(s-a)\beta_1(s, \cdot)z(s+t-a, 0, \cdot) ds, & \Gamma +a \ge t,
\end{cases}
\end{equation}
otherwise. Moreover
\begin{equation}\label{implicitformula}
y(t, a,\cdot)= S_2(T-t) y_T(T+a-t, \cdot),
\end{equation}
if $t \ge   T_2+ a$ (observe that in this case $T+a-t \le \bar a_2$) and
\begin{equation}\label{implicitformula1}
y(t,a, \cdot)=\begin{cases}
S_2(T-t) y_T(T+a-t, \cdot)\!+\int_a^{T+a-t}S_2(s-a)\beta_2(s, \cdot)y(s+t-a, 0, \cdot) ds, & \Gamma +a \le t \\
\int_a^AS_2(s-a)\beta_2(s, \cdot)y(s+t-a, 0, \cdot) ds, & \Gamma +a \ge t,
\end{cases}
\end{equation}
otherwise. Here  $(S_i(t))_{t \ge0}$ are the semigroups generated by the operators $A_i-\mu_{ii} Id$, $i=1,2$,  ($Id$ is the identity operator).
In particular, it results
\begin{equation}\label{z(0)}
z(t,0, \cdot):= S_1(T-t) z_T(T-t, \cdot)+ \int_0^{T-t}S_1(s)\mu_{21}(s+t, s,\cdot) y(s+t, s,\cdot)ds, 
\end{equation}
if $t \ge T-\bar a_1$ and
\begin{equation}\label{y(0)}
y(t,0, \cdot):= S_2(T-t) y_T(T-t, \cdot),
\end{equation}
if $t \ge T-\bar a_2$.
\vspace{0.3cm}

Now, we distinguish between the two cases $\bar a_1 \le \bar a_2$ and $\bar a_1>\bar a_2$.

{\underline{ $\bar a_1 \le \bar a_2$:}}
Using the same idea of \cite{f_JMPA}, we now prove that there exists a positive constant $C$ such that:
\begin{equation}\label{t=011}
 \int_{Q_{A,1}} \frac{y^2}{k_2} (T_2,a,x) dxda  \le C \int_{T-\frac{\bar a_2}{2}}^{T-\frac{\bar a_2}{4}} \int_{Q_{A,1}}\frac{y^2}{k_2} dxdadt,
\end{equation}
and
\begin{equation}\label{t=011z}
 \int_{Q_{A,1}} \frac{z^2}{k_1} (T_1,a,x) dxda  \le C\ds  \int_{T-\frac{\bar a_1}{2}}^{T-\frac{\bar a_1}{4}} \int_{Q_{A,1}}\left(\frac{z^2}{k_1}+ \frac{y^2}{k_2}\right) dxdadt.
\end{equation}

Indeed, define, for $\varsigma >0$, the functions $w= e^{\varsigma t }y$ and $W= e^{\varsigma t }z$. Then,
$w$ and $W$ satisfy, respectively, the problems
\begin{equation}\label{h=0'}
\begin{cases}
\ds \frac{\partial w}{\partial t} + \frac{\partial w}{\partial a}
+k_2(x)w_{xx}-(\mu_{22}(t, a, x)+ \varsigma) w =-\beta_2(a,x)w(t,0,x),& (t,x,a) \in  \tilde Q_2,
\\[5pt]
w(t,a,0)=w(t,a,1) =0, &(t,a) \in \tilde Q_{2,T,A},\\
  w(T,a,x) = e^{\varsigma T}y_T(a,x), &(a,x) \in Q_{A,1}, \\
  w(t,A,x)=0, & (t,x) \in \tilde Q_{2,T,1},
\end{cases}
\end{equation}
and
\begin{equation}\label{h=0'new}
\begin{cases}
\ds \frac{\partial W}{\partial t} + \frac{\partial W}{\partial a}
+k_1(x)W_{xx}-(\mu_{11}(t, a, x)+ \varsigma) W =-\mu_{21}w-\beta_1(a,x)W(t,0,x),& (t,x,a) \in  \tilde Q_1,
\\[5pt]
W(t,a,0)=W(t,a,1) =0, &(t,a) \in \tilde Q_{1,T,A},\\
  W(T,a,x) = e^{\varsigma T}z_T(a,x), &(a,x) \in Q_{A,1}, \\
  W(t,A,x)=0, & (t,x) \in \tilde Q_{1,T,1},
\end{cases}
\end{equation}
where $\tilde Q_i:= (T_i, T) \times Q_{A,1}$, $\tilde Q_{i,T,A}:= (T_i, T) \times (0,A)$ and $\tilde Q_{i, T,1}:= (T_i,T)\times (0,1)$, $i=1,2$.

Multiplying the equations of  \eqref{h=0'} and \eqref{h=0'new} by $\ds-\frac{w}{k_2}$ and  $\ds-\frac{W}{k_1}$ and integrating by parts on $Q_{2,t}:= (T_2,t) \times(0,A) \times(0,1)$ and $Q_{1,t}:= (T_1,t) \times(0,A) \times(0,1)$, respectively,  it results
\[
\begin{aligned}
&-\frac{1}{2}\int_{Q_{A,1}}\frac{1}{k_2} w^2(t,a,x) dxda + \frac{e^{2\varsigma T_2}}{2} \int_{Q_{A,1}} \frac{1}{k_2}y^2( T_2,a,x) dxda + \frac{1}{2} \int_{T_2}^t\int_0^1 \frac{1}{k_2} w^2(s,0,x) dx ds\\
&+ \varsigma \int_{Q_{2,t}}\frac{1}{k_2} w^2(s,a,x) dxdads \le \int_{Q_{2,t}}\frac{1}{k_2} \beta_2 w(s,0,x)wdxdads
\\
& \le \|\beta_2\|_{L^\infty(Q)}\frac{1}{\epsilon}\int_{Q_{2,t}}\frac{1}{k_2} w^2dxdads+ \epsilon A\|\beta_2\|_{L^\infty(Q)} \int_{T_2}^t\int_0^1\frac{1}{k_2}  w^2(s,0,x)dxds,
\end{aligned}
\]
and
\[
\begin{aligned}
&-\frac{1}{2}\int_{Q_{A,1}}\frac{1}{k_1} W^2(t,a,x) dxda + \frac{e^{2\varsigma   T_1}}{2} \int_{Q_{A,1}} \frac{1}{k_1}z^2(T_1,a,x) dxda + \frac{1}{2} \int_{T_1}^t\int_0^1 \frac{1}{k_1} W^2(s,0,x) dx ds\\
&+ \varsigma \int_{Q_{1,t}}\frac{1}{k_1} W^2(s,a,x) dxdads \le \int_{Q_{1,t}}\frac{1}{k_1} \beta_1 W(s,0,x)Wdxdads+ \int_{Q_{1,t}}\frac{\mu_{21}}{k_1} wWdxdads
\\
& \le \frac{1}{\epsilon}\left(\|\beta_1\|_{L^\infty(Q)}+ 1  \right)\int_{Q_{1,t}}\frac{1}{k_1} W^2dxdads+ \epsilon A\|\beta_1\|_{L^\infty(Q)} \int_{T_1}^t\int_0^1\frac{1}{k_1}  W^2(s,0,x)dxds\\
&+ \epsilon \|\mu_{21}^2\|_{L^\infty(Q)}\int_{Q_{1,t}}\frac{1}{k_1} w^2dxdads,
\end{aligned}
\]
for $\epsilon >0$. Choosing   $\ds\epsilon =\min\left\{ \frac{1}{2\|\beta_1\|_{L^\infty(Q)}A},  \frac{1}{2\|\beta_2\|_{L^\infty(Q)}A} \right\}$ and $\ds \varsigma =\max\left\{\frac{\|\beta_1\|_{L^\infty(Q)}+ 1}{\epsilon}, \frac{ \|\beta_2\|_{L^\infty(Q)}}{\epsilon}\right\}$, we have
\begin{equation}\label{Eq1}
\begin{aligned}
 \int_{Q_{A,1}}  \frac{1}{k_2} y^2(T_2,a,x) dxda  \le C\int_{Q_{A,1}}\frac{1}{k_2} w^2(t,a,x) dxda  \le C \int_{Q_{A,1}}\frac{1}{k_2}  y^2(t,a,x) dxda.
\end{aligned}
\end{equation}
and
\begin{equation}\label{Eq2}
\begin{aligned}
 \int_{Q_{A,1}}  \frac{1}{k_1} z^2(T_1,a,x) dxda  &\le C\int_{Q_{A,1}}\frac{1}{k_1} W^2(t,a,x) dxda + \epsilon \|\mu_{21}^2\|_{L^\infty(Q)}\int_{Q_{1,t}}\frac{1}{k_1} w^2dxdads\\
& \le C \left(\int_{Q_{A,1}}\frac{1}{k_1}  z^2(t,a,x) dxda+ \int_{Q_{1,t}}\frac{1}{k_2} y^2dxdads\right).
\end{aligned}
\end{equation}
Then, integrating \eqref{Eq1} over $\ds \left[T-\frac{\bar a_2}{2}, T-\frac{\bar a_2}{4} \right]$ and \eqref{Eq2} over $\ds \left[T-\frac{\bar a_1}{2}, T-\frac{\bar a_1}{4} \right]$,
 we have \eqref{t=011} and \eqref{t=011z}, respectively.
\vspace{0,5cm}

Now, take $\delta \in (0, A)$. 
By \eqref{t=011}, we have
\begin{equation}\label{t=0y}
\begin{aligned}
& \int_{Q_{A,1}} \frac{1}{k_2} y^2(T_2, a,x) dxda \le C   \int_{T-\frac{\bar a_2}{2}}^{T-\frac{\bar a_2}{4}}\left(\int_0^\delta + \int_\delta^A \right)\int_0^1 \frac{1}{k_2} y^2dxdadt,
\end{aligned}
\end{equation}
Consider the term $\ds \int_{T-\frac{\bar a_2}{2}}^{T-\frac{\bar a_2}{4}} \int_\delta^A\int_0^1\frac{1}{k_2} y^2(t,a,x) dxdadt$. By the Hardy-Poincar\'e inequality (see, for example, \cite{b}), one has 
\begin{equation}\label{terminenuovo11}
\begin{aligned}
 \int_0^1\frac{1}{k_2} y^2dx
& 
\le C\int_0^1 \frac{1}{x^2} y^2dx\le C \int_0^1 y_x^2 dx,
\end{aligned}\end{equation}
for a  strictly positive constant $C.$  Hence,
\begin{equation}\label{stimabo}
\begin{aligned}
\int_{T-\frac{\bar a_2}{2}}^{T-\frac{\bar a_2}{4}} \int_\delta^A\int_0^1 \frac{1}{k_2} y^2(t,a,x) dxdadt &\le C \int_{T-\frac{\bar a_2}{2}}^{T-\frac{\bar a_2}{4}} \int_\delta^A\int_0^1\tilde \Theta_2 y_x^2e^{2s\tilde\varphi_2} dxdadt \\
& +C \int_{T-\frac{\bar a_2}{2}}^{T-\frac{\bar a_2}{4}}  \int_\delta^A\int_0^1 \tilde \Theta_2^3\left( \frac{x}{k_2}\right)^2y^2e^{2s\tilde\varphi_2} dxdadt,
\end{aligned}
\end{equation}
where $\tilde \Theta_2$ is defined in \eqref{thetatilde} with $T_\gamma:= T-\bar a_2$, $T_\beta:= T$, $\gamma =0$ and $\tilde \varphi_2$ is the function  $\varphi_2$ associated to $\tilde \Theta_2$  according to \eqref{13}.
Thus, by Theorem \ref{Cor2} (or Theorem \ref{Carleman}) applied to $\bar Q_2:=(T-\bar a_2, T) \times (0, A)\times (0,1)$ and Remark \ref{remarkultimo},
\[
\int_{T-\frac{\bar a_2}{2}}^{T-\frac{\bar a_2}{4}}   \int_\delta^A\int_0^1 \frac{1}{k_2} y^2(t,a,x) dxdadt \le
C\left(\int_{\bar Q_2}\frac{h^{2}}{k_2}dxdadt+ \int_0^T \int_0^A\int_ \omega \frac{y^2}{k_2} dx dadt\right),
\]
where, in this case, $h(t,a,x):=-\beta_2(a,x)y(t,0,x)$.
Hence, by \eqref{y(0)}, since $t \ge T-\bar a_2$,
\begin{equation}\label{t=01}
\begin{aligned}
\int_{T-\frac{\bar a_2}{2}}^{T-\frac{\bar a_2}{4}}\!\! \int_\delta^A\!\!\int_0^1 \frac{1}{k_2} y^2dxdadt& \le C\|\beta_2\|^2_{L^\infty(Q)}\!\left(\int_{\bar Q_2}\frac{1}{k_2} y^2(t,0,x)dxdadt+ \int_0^T \int_0^A\!\!\int_ \omega \frac{1}{k_2} y^2 dx dadt\right),
\\
&\le C \left(\int_{\bar Q_2}\frac{1}{k_2} y^2_T(T-t,x)dxdt + \int_0^T \int_0^A\!\!\int_ \omega \frac{1}{k_2} y^2 dx dadt\right)
\\
&\le C \left(\int_0^{\bar a_2} \int_0^1\frac{y_T^2(\sigma,x) }{k_2}dxd\sigma + \int_0^T \int_0^A\!\!\int_ \omega \frac{1}{k_2} y^2 dx dadt\right),
\end{aligned}
\end{equation}
for a  strictly positive constant $C$.  By \eqref{t=0y} and \eqref{t=01}, \eqref{stima1} follows.

Now, we will prove \eqref{stima2}.
 Again, by \eqref{t=011z},
\begin{equation}\label{t=0z}
\begin{aligned}
 \int_{Q_{A,1}} \frac{z^2}{k_1} (T_1,a,x) dxda \le C  \int_{T-\frac{\bar a_1}{2}}^{T-\frac{\bar a_1}{4}}\left(\int_0^\delta + \int_\delta^A \right)\int_0^1\left(\frac{z^2}{k_1}+\frac{y^2}{k_2}\right) dxdadt.
\end{aligned}
\end{equation}
Proceeding as before, one can prove
\[
\begin{aligned}
\int_{T-\frac{\bar a_1}{2}}^{T-\frac{\bar a_1}{4}}\!\! \int_\delta^A\!\!\int_0^1 \frac{1}{k_2} y^2dxdadt& \le \int_{T-\frac{\bar a_2}{2}}^{T-\frac{\bar a_1}{4}}\!\! \int_\delta^A\!\!\int_0^1 \frac{1}{k_2} y^2dxdadt \\
&\le C \left(\int_0^{\bar a_2} \int_0^1\frac{y_T^2(a,x) }{k_2}dxda + \int_0^T \int_0^A\!\!\int_ \omega \frac{1}{k_2} y^2 dx dadt\right).
\end{aligned}
\]
Now, consider  $\ds \int_{T-\frac{\bar a_1}{2}}^{T-\frac{\bar a_1}{4}} \int_\delta^A\int_0^1\frac{1}{k_1} z^2dxdadt$. As before,  by Theorem \ref{Carleman} applied to $\bar Q_1:=\left(T-\ds\frac{3}{4}\bar a_1, T- \ds\frac{\bar a_1}{8}\right) \times (0, A)\times (0,1)$ and Remark \ref{remarkultimo}, replacing  $\tilde \varphi_2$,  $\tilde \Theta_2$ by $\tilde \varphi_1$, $\tilde \Theta_1$, respectively, where  $\tilde \varphi_1$  is the function  $\varphi_1$ associated to $\tilde \Theta_1$  according to \eqref{13}, and $\tilde \Theta_1$  is defined in \eqref{thetatilde} with $T_\gamma:= T-\ds\frac{3}{4}\bar a_1$, $T_\beta:=T- \ds\frac{\bar a_1}{8}$, $\gamma =0$, one has
\[
\int_{T-\frac{\bar a_1}{2}}^{T-\frac{\bar a_1}{4}}   \int_\delta^A\int_0^1 \frac{1}{k_1} z^2(t,a,x) dxdadt \le
C\left(\int_{\bar Q_1}\frac{f^{2}}{k_1}dxdadt+\int_{\bar Q_1}\frac{y^2}{k_1}dxdadt + \int_0^T \int_0^A\int_ \omega \frac{z^2}{k_1} dx dadt\right),
\]
where, in this case, $f(t,a,x):=- \beta_1(a,x)z(t,0,x)$. Hence, by \eqref{z(0)},
\[
\begin{aligned}
\int_{T-\frac{\bar a_1}{2}}^{T-\frac{\bar a_1}{4}}   \int_\delta^A\int_0^1 \frac{1}{k_1} z^2(t,a,x) dxdadt &\le C\left(\int_{\bar Q_1}\frac{y^2}{k_1}dxdadt + \int_{\bar Q_1}\frac{z^2(t,0,x)}{k_1}dxdadt  + \int_0^T \int_0^A\int_ \omega \frac{z^2}{k_1} dx dadt\right)\\
&\le  C\left(\int_{\bar Q_1}\frac{y^2}{k_2}dxdadt + \int_0^{\bar a_1} \int_0^1\frac{z_T^2(a,x) }{k_1}dxda	\right) +\\
&\left( \int_{T-\bar a_1}^T\int_0^1 \left(\int_0^{T-t} \frac{1}{k_1}y^2(s+t,s,x)ds\right)dxdt+ \int_0^T \int_0^A\int_ \omega \frac{z^2}{k_1} dx dadt\right).
\end{aligned}
\]
It remains to estimate $ \ds\int_{\bar Q_1}\frac{y^2}{k_2}dxdadt$ and $\ds  \int_{T-\bar a_1}^T\int_0^1 \left(\int_0^{T-t} \frac{1}{k_1}y^2(s+t,s,x)ds\right)dxdt$.
First of all, we prove
\begin{equation}\label{stima3}
\int_{\bar Q_1}\frac{y^2}{k_2}dxdadt\le  C\left( \int_{T- \bar a_1}^{T} \int_0^\delta\int_0^1 \frac{1}{k_2} y^2(t,a,x) dxdadt + \int_0^{\bar a_2} \int_0^1\frac{y_T^2(a,x) }{k_2}dxda + \int_0^T \int_0^A\!\!\int_ \omega \frac{1}{k_2} y^2 dx dadt\right).
\end{equation}
 Again we have
\[
\int_{\bar Q_1}\frac{y^2}{k_2}dxdadt =  \int_{T- \frac{3}{4}\bar a_1}^{T-\frac{\bar a_1}{8}}\left(\int_0^\delta + \int_\delta^A \right)\int_0^1\frac{y^2}{k_2}dxdadt \le \int_{T-\bar a_1}^{T}\int_0^\delta \int_0^1\frac{y^2}{k_2}dxdadt  +  \int_{T- \frac{3}{4}\bar a_1}^{T-\frac{\bar a_1}{8}}\int_\delta^A\int_0^1\frac{y^2}{k_2}dxdadt.
\]
By \eqref{terminenuovo11} and proceeding as before,
\[
\begin{aligned}
 \int_{T-\frac{3}{4} \bar a_1}^{T-\frac{\bar a_1}{8}} \int_\delta^A\int_0^1 \frac{1}{k_2} y^2(t,a,x) dxdadt &\le C \int_{T- \frac{3}{4}\bar a_1}^{T-\frac{\bar a_1}{8}} \int_\delta^A\int_0^1\bar\Theta_2 y_x^2e^{2s\bar\varphi_2} dxdadt \\
& +C \int_{T- \frac{3}{4}\bar a_1}^{T-\frac{\bar a_1}{8}} \int_\delta^A\int_0^1 \bar \Theta_2^3\left( \frac{x}{k_2}\right)^2y^2e^{2s\bar\varphi_2} dxdadt,
\end{aligned}
\]
where $\bar \Theta_2$ is defined in \eqref{thetatilde} with $T_\gamma:= T-\bar a_1$, $T_\beta:= T$, $\gamma =0$ and $\bar \varphi_2$ is the function  $\varphi_2$ associated to $\bar \Theta_2$  according to \eqref{13}.
Thus, by Theorem \ref{Cor2} (or Theorem \ref{Carleman}) applied to $ Q_2:=\left(T-\bar a_1, T\right) \times (0, A)\times (0,1)$ ,
\[
\int_{T- \frac{3}{4}\bar a_1}^{T- \frac{\bar a_1}{8}} \int_\delta^A\int_0^1 \frac{1}{k_2} y^2(t,a,x) dxdadt \le
C\left(\int_{ Q_2}\frac{g^{2}}{k_2}dxdadt+ \int_0^T \int_0^A\int_ \omega \frac{y^2}{k_2} dx dadt\right),
\]
where, in this case, $g(t,a,x):=-\beta_2(a,x)y(t,0,x)$.
Hence, by \eqref{y(0)},
\begin{equation}\label{t=01'}
\begin{aligned}
\int_{T- \frac{3}{4}\bar a_1}^{T-\frac{\bar a_1}{8}} \!\! \int_\delta^A\!\!\int_0^1 \frac{1}{k_2} y^2dxdadt& \le C\|\beta_2\|^2_{L^\infty(Q)}\!\left(\int_{ Q_2}\frac{1}{k_2} y^2(t,0,x)dxdadt+ \int_0^T \int_0^A\!\!\int_ \omega \frac{1}{k_2} y^2 dx dadt\right),
\\
&\le C \left(\int_0^{\bar a_1} \int_0^1\frac{y_T^2(a,x) }{k_2}dxda + \int_0^T \int_0^A\!\!\int_ \omega \frac{1}{k_2} y^2 dx dadt\right)\\
&\le C \left(\int_0^{\bar a_2} \int_0^1\frac{y_T^2(a,x) }{k_2}dxda + \int_0^T \int_0^A\!\!\int_ \omega \frac{1}{k_2} y^2 dx dadt\right),
\end{aligned}
\end{equation}
for a  strictly positive constant $C$. Thus \eqref{stima3} follows. Finally, we prove that there exists a positive constant $C$, such that
\[
  \int_{T-\bar a_1}^T\int_0^1 \left(\int_0^{T-t} \frac{1}{k_1}y^2(s+t,s,x)ds\right)dxdt \le C\int_0^{\bar a_2} \int_0^1 \frac{1}{k_2}y_T^2(s,x)dsdx.
\]
Indeed, since $t \in (T-\bar a_1,T)$ and $\bar a_1 \le \bar a_2$, it follows  $t \ge T-\bar a_1 \ge T-\bar a_2$. Thus, by the implicit formula of $y$, it follows
that
\[
y(s+t,s,x)= S_2(T-s-t) y_T(T-t,x).
\]
Hence
\[
\begin{aligned}
 & \int_{T-\bar a_1}^T\int_0^1\int_0^{T-t} \frac{1}{k_1}y^2(s+t,s,x)dsdxdt \le C   \int_{T-\bar a_1}^T\int_0^1\int_0^{T-t} \frac{1}{k_1}y_T^2(T-t,x)dsdxdt\\
&\le C\int_0^{\bar a_1} \int_0^1 \frac{1}{k_2}y_T^2(s,x)dsdx \le  C\int_0^{\bar a_2} \int_0^1 \frac{1}{k_2}y_T^2(s,x)dsdx.
\end{aligned}
\]

{\underline{ $\bar a_1 > \bar a_2$:}} As before, one can prove \eqref{Eq1}, \eqref{Eq2}, \eqref{t=01} and hence \eqref{stima1}.  Now, we prove \eqref{stima2}. As for the case $\bar a_1 \le \bar a_2$,  we multiply
 the equation of  \eqref{h=0'new} by  $\ds-\frac{W}{k_1}$ and integrate by parts on $\bar Q_{2,t}:= (T_1,t) \times(0,A) \times(0,1)$. Hence, one can obtain

\begin{equation}\label{Eq2n}
\begin{aligned}
 \int_{Q_{A,1}}  \frac{1}{k_1} z^2(T_1,a,x) dxda  \le C \left(\int_{Q_{A,1}}\frac{1}{k_1}  z^2(t,a,x) dxda+ \int_{\bar Q_{2,t}}\frac{1}{k_1} y^2dxdads\right).
\end{aligned}
\end{equation}
Now, take $\delta \in (0, A)$.  Then, integrating \eqref{Eq2n} over $\ds \left[T-\frac{\bar a_2}{4}, T-\frac{\bar a_2}{8} \right]$,
 we have
 \begin{equation}\label{t=011n}
\begin{aligned}
 \int_{Q_{A,1}} \frac{z^2}{k_1} (T_1,a,x) dxda  &\le C\ds  \int_{T-\frac{\bar a_2}{4}}^{T-\frac{\bar a_2}{8}} \int_{Q_{A,1}}\left(\frac{z^2}{k_1}+ \frac{y^2}{k_1}\right) dxdadt
\\
&\le C \int_{T-\frac{\bar a_2}{4}}^{T-\frac{\bar a_2}{8}}\left(\int_0^\delta + \int_\delta^A \right)\int_0^1\left(\frac{z^2}{k_1}+\frac{y^2}{k_2}\right) dxdadt.
\end{aligned}
\end{equation}
Proceeding as in \eqref{stimabo}, we have
\[
\begin{aligned}
\int_{T-\frac{\bar a_2}{4}}^{T-\frac{\bar a_2}{8}}\!\! \int_\delta^A\!\!\int_0^1 \frac{1}{k_2} y^2dxdadt& 
&\le C \left(\int_0^{\bar a_2} \int_0^1\frac{y_T^2(a,x) }{k_2}dxda + \int_0^T \int_0^A\!\!\int_ \omega \frac{1}{k_2} y^2 dx dadt\right),
\end{aligned}
\]
where
$\tilde \Theta_2$ and $\tilde \varphi_2$ are as in \eqref{stimabo}, hence $T_\gamma:= T-\bar a_2$, $T_\beta:= T$ and $\gamma =0$. Observe that we have $\ds t \ge T-\frac{\bar a_2}{4} \ge T-\bar a_2$, hence we can apply the implicit formula \eqref{y(0)}.

Now, consider  $\ds \int_{T-\frac{\bar a_2}{4}}^{T-\frac{\bar a_2}{8}} \int_\delta^A\int_0^1\frac{1}{k_1} z^2dxdadt$. As before,  by Theorem \ref{Carleman} applied to $\bar Q_1:=\left(T-\ds \frac{\bar a_2}{2}, T- \ds\frac{\bar a_2}{10}\right) \times (0, A)\times (0,1)$, replacing  $\tilde \varphi_2$,  $\tilde \Theta_2$ by $\tilde \varphi_1$, $\tilde \Theta_1$, respectively, where  $\tilde \varphi_1$  is the function  $\varphi_1$ associated to $\tilde \Theta_1$  according to \eqref{13}, and $\tilde \Theta_1$  is defined in \eqref{thetatilde} with $T_\gamma:= T-\ds \frac{ \bar a_2}{2}$, $T_\beta:=T- \ds\frac{\bar a_2}{10}$, $\gamma =0$, one has
\[
\int_{T-\frac{\bar a_2}{4}}^{T-\frac{\bar a_2}{8}}  \int_\delta^A\int_0^1 \frac{1}{k_1} z^2(t,a,x) dxdadt \le
C\left(\int_{\bar Q_1}\frac{f^{2}}{k_1}dxdadt+ \int_{\bar Q_1}\frac{y^2}{k_1}dxdadt + \int_0^T \int_0^A\int_ \omega \frac{z^2}{k_1} dx dadt\right),
\]
where again $f(t,a,x):= - \beta_1(a,x)z(t,0,x)$. Hence
\begin{equation}\label{tondino}
\begin{aligned}
&\int_{T-\frac{\bar a_2}{4}}^{T-\frac{\bar a_2}{8}} \int_\delta^A\int_0^1 \frac{1}{k_1} z^2(t,a,x) dxdadt \le C\left(\int_{\bar Q_1}\frac{y^2}{k_1}dxdadt + \int_{\bar Q_1}\frac{z^2(t,0,x)}{k_1}dxdadt  + \int_0^T \int_0^A\int_ \omega \frac{z^2}{k_1} dx dadt\right)\\
&\le  C\left(\int_{\bar Q_1}\frac{y^2}{k_2}dxdadt + \int_0^{\bar a_1} \int_0^1\frac{z_T^2(a,x) }{k_1}dxda + \int_{T-\frac{\bar a_2}{2}}^{T-\frac{\bar a_2}{10}}\int_0^1 \left(\int_0^{T-t} \frac{1}{k_1}y^2(s+t,s,x)ds\right)dxdt\right)
\\
&+ C\int_0^T \int_0^A\int_ \omega \frac{z^2}{k_1} dx dadt.
\end{aligned}
\end{equation}
 Again we have
\begin{equation}\label{star}
\begin{aligned}
\int_{\bar Q_1}\frac{y^2}{k_2}dxdadt &= \int_{T-\frac{\bar a_2}{2}}^{T-\frac{\bar a_2}{10}}\left(\int_0^\delta + \int_\delta^A \right)\int_0^1\frac{y^2}{k_2}dxdadt \\
&\le C \left(\int_0^T\int_0^\delta \int_0^1\frac{y^2}{k_2}dxdadt
+\int_0^{\bar a_2} \int_0^1\frac{y_T^2(a,x) }{k_2}dxda + \int_0^T \int_0^A\!\!\int_ \omega \frac{1}{k_2} y^2 dx dadt\right).\end{aligned}
\end{equation}
Now, it remains to estimate
\[
\begin{aligned}
 \int_{T-\frac{\bar a_2}{2}}^{T-\frac{\bar a_2}{10}}\int_0^1 \left(\int_0^{T-t} \frac{1}{k_1}y^2(s+t,s,x)ds\right)dxdt.
\end{aligned}
\]
Since $t \in \left[ T-\frac{\bar a_2}{2}, T-\frac{\bar a_2}{10}\right]$, it holds $\frac{\bar a_2}{10} \le T-t$, thus one can divide the integral
in the following way
\[
\begin{aligned}
 \int_{T-\frac{\bar a_2}{2}}^{T-\frac{\bar a_2}{10}}\int_0^1 \left(\int_0^{T-t} \frac{1}{k_1}y^2(s+t,s,x)ds\right)dxdt &\le  \int_{T-\frac{\bar a_2}{2}}^{T-\frac{\bar a_2}{10}}\int_0^{\frac{\bar a_2}{10} }\int_0^1\frac{1}{k_2}y^2(s+t,s,x)dxdsdt \\
& + \int_{T-\frac{\bar a_2}{2}}^{T-\frac{\bar a_2}{10}} \int_{\frac{\bar a_2}{10} }^{T-t}\int_0^1  \frac{1}{k_2}y^2(s+t,s,x)dxdsdt.
\end{aligned}
\]
Now, by \eqref{terminenuovo11} and Theorem \ref{Cor2} (or Theorem \ref{Carleman}) applied to $y$ in $\bar Q_3:=\left(T-\bar a_2, T - \frac{\bar a_2}{12}\right) \times (0, T-t)\times (0,1)$, replacing  $\varphi_2$,  $ \Theta_2$ by $\tilde \varphi_3$, $\tilde \Theta_3$, respectively, where  $\tilde \varphi_3$  is the function  $\varphi_2$ associated to $\tilde \Theta_3$  according to \eqref{13}, and $\tilde \Theta_3$  is defined in \eqref{thetatilde} with $T_\gamma:= T-\bar a_2$, $T_\beta:=T-\frac{\bar a_2}{12}$, $\gamma =0$, one has
\[
\begin{aligned}
 & \int_{T-\frac{\bar a_2}{2}}^{T-\frac{\bar a_2}{10}} \int_{\frac{\bar a_2}{10} }^{T-t}\int_0^1  \frac{1}{k_2}y^2(s+t,s,x)dxdsdt \le 
 \int_{T-\frac{\bar a_2}{2}}^{T-\frac{\bar a_2}{10}} \int_{\frac{\bar a_2}{10} }^{T-t}\int_0^1  y_x^2(s+t,s,x)dxdsdt\\
& \le \int_{\bar Q_3}\frac{h^2(s+t,s,x)}{k_2} dxds dt+ \int_{T-\bar a_2}^T\int_0^{T-t} \int_\omega\frac{y^2(s+t,s,x)}{k_2} dxdsdt
\end{aligned}
\]
where, in this case, $h(s+t,s,x):=-\beta_2(s,x)y(s+t,0,x)$. Moreover, $s+t \ge t \ge T-\bar a_2$; thus, by \eqref{y(0)}, one has
\[
\begin{aligned}
\int_{\bar Q_3}\frac{h^2(s+t,s,x)}{k_2} dxds dt&\le C \int_{\bar Q_3}\frac{y^2(s+t,0,x)}{k_2} dxds dt \le \int_{\bar Q_3}\frac{y_T^2(T-s-t,x)}{k_2} dxds dt \\
&\le C\int_{T-\bar a_2}^T\int_0^{T-t}\int_0^1\frac{y_T^2(\tau,x)}{k_2} dxd\tau dt \le C\int_0^{\bar a_2}\int_0^1\frac{y_T^2(\tau,x)}{k_2} dxd\tau.
\end{aligned}\]
Analogously, 
\[
 \int_{T-\frac{\bar a_2}{2}}^T\int_0^{T-t} \int_\omega\frac{y^2(s+t,s,x)}{k_2} dxdsdt\le C  \int_{T-\bar a_2}^T\int_0^{T-t} \int_\omega\frac{y_T^2(T-t,x)}{k_2} dxdsdt \le  C\int_0^{\bar a_2}\int_0^1\frac{y_T^2(\tau,x)}{k_2} dxd\tau.
\]
Finally, since $t \ge T-\frac{\bar a_2}{2}\ge T-\bar a_2$, one has $s+t \ge T-\bar a_2 +s$; thus, by \eqref{implicitformula} obtaining
\begin{equation}\label{startondino}
\begin{aligned}
 \int_{T-\frac{\bar a_2}{2}}^{T-\frac{\bar a_2}{10}}\int_0^{\frac{\bar a_2}{10} }\int_0^1\frac{1}{k_2}y^2(s+t,s,x)dxdsdt&\le C   \int_{T-\frac{\bar a_2}{2}}^{T-\frac{\bar a_2}{10}}\int_0^{\frac{\bar a_2}{10} }\int_0^1\frac{1}{k_2}y_T^2(T-t,x)dxdsdt
\\
&\le C \int_0^{\bar a_2}\int_0^1\frac{y_T^2(\tau,x)}{k_2} dxd\tau.
\end{aligned}
\end{equation}
Hence, \eqref{stima2} holds also in the case $\bar a_1 >\bar a_2$.
\end{proof}

We underline that in this paper we improve \eqref{stima1} given in \cite{f_JMPA} for a single equation. Indeed here we do not require that $T<A$ and the proof is simpler. The same improvement holds also for the next result which generalizes the previous theorem.
\begin{Theorem}\label{CorOb1'}Assume Hypotheses $\ref{ratesAss}$, $\ref{Assw}$, $\ref{AssP}$ and  \ref{conditionbeta}. Then, for every $\delta \in (\gamma,A)$, where $\gamma = \max \{\bar a_1, \bar a_2\}$,
there exists a strictly positive constant $C=C(\delta)$ such that every
solution $(z,y)\in L^2\big(Q_{T,A}; \mathbb{K}\big) \cap H^1\big(0, T; H^1(0,A;\mathbb{H})\big)$ of \eqref{adjoint}
satisfies
\begin{equation}\label{ribo}
\int_{Q_{A,1}} \frac{1}{k_2} y^2(T-\bar a_2, a,x) dxda  \le 
 C\left( \int_0^\delta \int_0^1\frac{1}{k_2} y_T^2(a,x)dxda+ \int_{Q_{T,A}}\int_ \omega\frac{1}{k_2} y^2 dx dadt\right)
\end{equation}
and
\begin{equation}\label{ribo1}
\begin{aligned}
\int_{Q_{A,1}} \frac{1}{k_1} z^2(T-\bar a_1, a,x) dxda & \le 
 C\left( \int_0^\delta \int_0^1\left(\frac{z_T^2(a,x)}{k_1} +\frac{y_T^2(a,x) }{k_2} \right)dxda+ \int_{Q_{T,A}}\int_ \omega\left(\frac{z^2 }{k_1} + \frac{y^2}{k_2}  \right) dx dadt\right).
\end{aligned}
\end{equation}
\end{Theorem}

\begin{proof} 
As in the previous theorem, we distinguish between the case $\bar a_1 \le \bar a_2$ and $\bar a_1 > \bar a_2$.

First of all, assume $\bar a_1 \le \bar a_2$. Again \eqref{Eq1} and \eqref{Eq2} hold.
Then, integrating  \eqref{Eq1} over  $\left[T-\ds\frac{\bar a_2}{2}, T-\ds\frac{\bar a_2}{4}\right]$   and   \eqref{Eq2} over $\left[T-\ds\frac{\bar a_1}{2}, T-\ds\frac{\bar a_1}{4}\right]$, we have for all $\delta >\gamma$:
\begin{equation}\label{t=041new}
 \int_{Q_{A,1}} \frac{1}{k_2} y^2(T-\bar a_2,a,x) dxda  \le C\int_{T-\frac{\bar a_2}{2}}^{T-\frac{\bar a_2}{4}} \left(\int_0^{\delta - \bar a_2} + \int_{\delta - \bar a_2}^A \right)\int_0^1 \frac{1}{k_2} y^2(t,a,x) dxdadt
\end{equation}
and
\begin{equation}\label{t=041new'}
 \int_{Q_{A,1}} \frac{1}{k_1}z^2(T-\bar a_1a,x) dxda  \le C\int_{T-\frac{\bar a_1}{2}}^{T-\frac{\bar a_1}{4}} \left(\int_0^{\delta - \bar a_1} + \int_{\delta - \bar a_1}^A \right)\int_0^1\left( \frac{z^2}{k_1} + \frac{y^2}{k_2}\right)dxdadt.
\end{equation}
Proceeding as before, one can prove \eqref{t=01}. Thus, using Theorem \ref{Cor2}, we can prove
\begin{equation}\label{bo2}
\begin{aligned}
\int_{T-\frac{\bar a_2}{2}}^{T-\frac{\bar a_2}{4}} \int_{\delta - \bar a_2}^A \!\!\int_0^1\frac{1}{k_2} y^2(t,a,x) dxdadt &\le C\int_0^{\bar a_2}\int_0^1\!\! \frac{1}{k_2} y^2_T(a,x)dxda+C \int_0^T \!\!\int_0^A\!\!\int_ \omega \frac{1}{k_2} y^2 dx dadt.
\end{aligned}
\end{equation}
 It remains to estimate
\[
\int_{T-\frac{\bar a_2}{2}}^{T-\frac{\bar a_2}{4}}\int_0^{\delta - \bar a_2}\int_0^1 \frac{1}{k_2} y^2(t,a,x) dxdadt.
\]
Observe that $t \ge T-\ds \frac{\bar a_2}{2} \ge T-\bar a_2$ and $a \in (0, \delta -\bar a_2)$. Thus $T-t \le \bar a_2 \le \delta-a \le A-a$ and, by the first formula in \eqref{implicitformula1}, we have 
\[
y(t,a, \cdot)=
S_2(T-t) y_T(T+a-t, \cdot)\!+\int_a^{T+a-t}S_2(s-a)\beta_2(s, \cdot)y(s+t-a, 0, \cdot) ds.
\]
It follows:
\begin{equation}\label{secondanew}
\begin{aligned}
&
\int_{T-\frac{\bar a_2}{2}}^{T-\frac{\bar a_2}{4}}\int_0^{\delta - \bar a_2} \int_0^1  \frac{1}{k_2}y^2(t,a,x) dxdadt\\
&\le C\int_{T-\frac{\bar a_2}{2}}^{T-\frac{\bar a_2}{4}}\int_0^{\delta - \bar a_2}\int_0^1 \frac{1}{k_2}y^2_T(T+a-t,x)dxdadt \\
&+C\int_{T-\frac{\bar a_2}{2}}^{T-\frac{\bar a_2}{4}}\int_0^{\delta - \bar a_2}\int_0^1  \left(\int_a^{T+a-t}\frac{1}{k_2}y^2(s+t-a,0,x)ds \right)dxdadt
\end{aligned}
\end{equation}
\[
\begin{aligned}
&\qquad =C\int_{\frac{\bar a_2}{4}}^{\frac{\bar a_2}{2}}\int_{0}^{\delta -\bar a_2} \int_0^1 \frac{1}{k_2}y^2_T(a+z,x)dxda dz \\
&\qquad +C\int_{T-\frac{\bar a_2}{2}}^{T-\frac{\bar a_2}{4}}\int_0^{\delta - \bar a_2} \int_0^1  \left(\int_{-a}^{T-a-t}\frac{1}{k_2}y_T^2(a+z,x)dz \right)dxdadt
\\
&\qquad  \le C\int_{\frac{\bar a_2}{4}}^{\delta-\frac{\bar a_2}{2}} \int_0^1 \frac{1}{k_2}y^2_T(\sigma,x)dxd\sigma dz \\
&\qquad +C\int_{T-\frac{\bar a_2}{2}}^{T-\frac{\bar a_2}{4}}\int_0^{\delta - \bar a_2}\int_0^1  \left(\int_0^{T-t}\frac{1}{k_2}y_T^2(\sigma,x)d\sigma \right)dxdadt\\
&\qquad  \le C\int_{0}^{\delta} \int_0^1 \frac{1}{k_2}y^2_T(\sigma,x)dxd\sigma +C \int_{T-\bar a_2}^{T-\frac{\bar a_2}{4}}\int_0^{\delta -\bar a_2} \int_0^1  \left(\int_0^{\bar a_2}\frac{1}{k_2}y_T^2(\sigma,x)d\sigma \right)dxdadt\\
&\qquad  \le C\int_0^\delta\int_0^1\frac{1}{k_2}y_T^2(\sigma,x)d\sigma dx.
\end{aligned}
\]
By \eqref{t=041new}-\eqref{secondanew}, \eqref{ribo} follows. Now, we estimate \eqref{t=041new'}.
Proceeding as in the previous theorem and in \eqref{secondanew}, one has
\[
\begin{aligned}
\int_{T-\frac{\bar a_1}{2}}^{T-\frac{\bar a_1}{4}}\!\! \int_{\delta- \bar a_1}^A\!\!\int_0^1 \frac{1}{k_2} y^2dxdadt \le C \left(\int_0^{\bar a_2} \int_0^1\frac{y_T^2(a,x) }{k_2}dxda + \int_0^T \int_0^A\!\!\int_ \omega \frac{1}{k_2} y^2 dx dadt\right),
\end{aligned}
\]
and
\[\begin{aligned}
\int_{T-\frac{\bar a_1}{2}}^{T-\frac{\bar a_1}{4}}\int_0^{\delta - \bar a_1} \int_0^1  \frac{1}{k_2}y^2(t,a,x) dxdadt \le C\int_0^\delta\int_0^1\frac{1}{k_2}y_T^2(\sigma,x)d\sigma dx.
\end{aligned}
\]
Moreover, by \eqref{star},
\[
\begin{aligned}
&\int_{T-\frac{\bar a_1}{2}}^{T-\frac{\bar a_1}{4}}   \int_{\delta-\bar a_1}^A\int_0^1 \frac{1}{k_1} z^2(t,a,x) dxdadt \le  C \int_{T- \bar a_1}^{T} \int_0^{\delta- \bar a_1}\int_0^1 \frac{1}{k_2} y^2(t,a,x) dxdadt +C\int_0^{\bar a_2} \int_0^1\frac{y_T^2(a,x) }{k_2}dxda \\&+C \int_{Q_{T,A}}\int_ \omega \frac{1}{k_2} y^2 dx dadt+
\int_0^{\bar a_1} \int_0^1\frac{z_T^2(a,x) }{k_1}dxda + \int_{Q_{T,A}}\int_ \omega \frac{z^2}{k_1} dx dadt
\\&
+ \int_{T-\bar a_1}^T\int_0^1 \left(\int_0^{T-t} \frac{1}{k_1}y^2(s+t,s,x)ds\right)dxdt\\
&\le C\int_0^{\delta} \int_0^1\frac{y_T^2(a,x) }{k_2}dxda +C \int_0^T \int_0^A\!\!\int_ \omega \frac{1}{k_2} y^2 dx dadt+
\int_0^{\delta} \int_0^1\frac{z_T^2(a,x) }{k_1}dxda + \int_{Q_{T,A}}\int_ \omega \frac{z^2}{k_1} dx dadt.
\end{aligned}
\]
Now, we consider the integral
\begin{equation}
 \int_{T-\frac{\bar a_1}{2}}^{T-\frac{\bar a_1}{4}} \int_0^{\delta - \bar a_1}\int_0^1 \frac{1}{k_1} z^2dxdadt.
\end{equation}
Observe that $t \ge T-\ds \frac{\bar a_1}{2} \ge T-\bar a_1$ and $a \in (0, \delta -\bar a_1)$. Thus $T-t \le \bar a_1 \le \delta-a \le A-a$ and, by the first formula in \eqref{implicitformula3}, the first formula in \eqref{implicitformula1} and by \eqref{z(0)}, it follows
\begin{equation}\label{secondanew1}
\begin{aligned}
&
\int_{T-\frac{\bar a_1}{2}}^{T-\frac{\bar a_1}{4}}\int_0^{\delta - \bar a_1} \int_0^1  \frac{1}{k_1}z^2(t,a,x) dxdadt\\
&\le C\int_{T-\frac{\bar a_1}{2}}^{T-\frac{\bar a_1}{4}}\int_0^{\delta - \bar a_1}\int_0^1 \frac{1}{k_1}z^2_T(T+a-t,x)dxdadt \\
&+C\int_{T-\frac{\bar a_1}{2}}^{T-\frac{\bar a_1}{4}}\int_0^{\delta - \bar a_1}\int_0^1  \left(\int_a^{T+a-t}\frac{1}{k_1}z^2(s+t-a,0,x)ds \right)dxdadt\\
&+C\int_{T-\frac{\bar a_1}{2}}^{T-\frac{\bar a_1}{4}}\int_0^{\delta - \bar a_1}\int_0^1  \left(\int_a^{T+a-t}\frac{1}{k_1}y^2(s+t-a,s,x)ds \right)dxdadt\\
\end{aligned}
\end{equation}
\[
\begin{aligned}
&\qquad \quad \quad =C\int_{\frac{\bar a_1}{4}}^{\frac{\bar a_1}{2}}\int_{0}^{\delta -\bar a_1} \int_0^1 \frac{1}{k_1}z^2_T(a+z,x)dxda dz \\
&\qquad \quad \quad+C\int_{T-\frac{\bar a_1}{2}}^{T-\frac{\bar a_1}{4}}\int_0^{\delta - \bar a_1} \int_0^1  \left(\int_{-a}^{T-a-t}\frac{1}{k_1}z_T^2(a+z,x)dz \right)dxdadt\\
& \qquad \quad \quad+C\int_{T-\frac{\bar a_1}{2}}^{T-\frac{\bar a_1}{4}}\int_0^{\delta - \bar a_1}\int_0^1  \left(\int_a^{T+a-t}\frac{1}{k_1}y_T^2(T-t+a,x)ds \right)dxdadt\\
&\qquad \quad \quad+C\int_{T-\frac{\bar a_1}{2}}^{T-\frac{\bar a_1}{4}}\int_0^{\delta - \bar a_1}\int_0^1  \left(\int_a^{T+a-t}\frac{1}{k_1}y^2(s+t-a,0,x)ds \right)dxdadt
\\
& \qquad \quad \quad\le C\int_{\frac{\bar a_1}{4}}^{\delta-\frac{\bar a_1}{2}} \int_0^1 \frac{1}{k_1}z^2_T(\sigma,x)dxd\sigma dz +C\int_{T-\frac{\bar a_1}{2}}^{T-\frac{\bar a_1}{4}}\int_0^{\delta - \bar a_1}\int_0^1  \left(\int_0^{T-t}\frac{1}{k_1}z_T^2(\sigma,x)d\sigma \right)dxdadt\\
&\qquad \quad \quad+C\int_{T-\frac{\bar a_1}{2}}^{T-\frac{\bar a_1}{4}}\int_0^{\delta - \bar a_1}\int_0^1 \frac{1}{k_1}y_T^2(T-t+a,x)dxdadt\\&
\qquad \quad \quad+C\int_{T-\frac{\bar a_1}{2}}^{T-\frac{\bar a_1}{4}}\int_0^{\delta - \bar a_1}\int_0^1  \left(\int_a^{T+a-t}\frac{1}{k_1}y_T^2(T-s-t+a,x)ds \right)dxdadt\\
&\qquad \quad \quad\le
C\int_0^{\delta} \int_0^1 \frac{1}{k_1}z^2_T(\sigma,x)dxd\sigma dz +C\int_{T-\frac{\bar a_1}{2}}^{T-\frac{\bar a_1}{4}}\int_0^{\delta - \bar a_1}\int_0^1  \left(\int_0^{\bar a_1}\frac{1}{k_1}z_T^2(\sigma,x)d\sigma \right)dxdadt\\
&\qquad \quad \quad+C\int_0^{\delta}\int_0^1 \frac{1}{k_1}y_T^2(\sigma,x)dxd\sigma+ C\int_{T-\frac{\bar a_1}{2}}^{T-\frac{\bar a_1}{4}}\int_0^{\delta - \bar a_1} \int_0^1  \left(\int_{-a}^{T-a-t}\frac{1}{k_1}y_T^2(a+z,x)dz \right)dxdadt\\
&\qquad \quad \quad \le C\int_0^{\delta} \int_0^1 \frac{1}{k_1}z^2_T(\sigma,x)dxd\sigma dz+C\int_0^{\delta}\int_0^1 \frac{1}{k_1}y_T^2(\sigma,x)dxd\sigma\\
&\qquad \quad \quad+ C\int_{T-\frac{\bar a_1}{2}}^{T-\frac{\bar a_1}{4}}\int_0^{\delta - \bar a_1} \int_0^1  \left(\int_0^{T-t}\frac{1}{k_1}y_T^2(\sigma,x)dz \right)dxdadt\\
&\qquad \quad \quad \le C\int_0^{\delta} \int_0^1 \frac{1}{k_1}z^2_T(\sigma,x)dxd\sigma dz+ C\int_0^\delta\int_0^1\frac{1}{k_2}y_T^2(\sigma,x)d\sigma dx.
\end{aligned}
\]
Now, we assume $\bar a_1 > \bar a_2$:

As before, one can prove \eqref{t=041new}, \eqref{bo2}, \eqref{secondanew} and hence \eqref{ribo}.  Now, we prove \eqref{ribo1}. Again \eqref{Eq2n} holds with $\bar Q_{2,t}:= \left(\ds T- \frac{\bar a_2}{4},t\right) \times(0,A) \times(0,1)$. 
Now, take $\delta \in (\gamma, A)$.  Then, integrating \eqref{Eq2n} over $\ds \left[T-\frac{\bar a_2}{4}, T-\frac{\bar a_2}{8} \right]$,
 we have
 \begin{equation}\label{t=011new}
\begin{aligned}
 \int_{Q_{A,1}} \frac{z^2}{k_1} (T_1,a,x) dxda \le C \int_{T-\frac{\bar a_2}{4}}^{T-\frac{\bar a_2}{8}}\left(\int_0^{\delta-\bar a_1} + \int_ {\delta-\bar a_1}^A \right)\int_0^1\left(\frac{z^2}{k_1}+\frac{y^2}{k_2}\right) dxdadt.
\end{aligned}
\end{equation}
Proceeding as in \eqref{t=01}, we have
\[
\begin{aligned}
\int_{T-\frac{\bar a_2}{4}}^{T-\frac{\bar a_2}{8}}\!\! \int_{\delta-\bar a_1}^A\!\!\int_0^1 \frac{1}{k_2} y^2dxdadt& 
&\le C \left(\int_0^{\bar a_2} \int_0^1\frac{y_T^2(a,x) }{k_2}dxda + \int_0^T \int_0^A\!\!\int_ \omega \frac{1}{k_2} y^2 dx dadt\right),
\end{aligned}
\]
(observe that \eqref{stimabo} holds with
$\tilde \Theta_2$ and $\tilde \varphi_2$  as in \eqref{stimabo}, and $T_\gamma:= T-\bar a_2$, $T_\beta:= T$ and $\gamma =0$). Moreover, we have $\ds t \ge T-\frac{\bar a_2}{4} \ge T-\bar a_2$, hence we can apply the implicit formula \eqref{y(0)}.
Now, consider the term
\[
\int_{T-\frac{\bar a_2}{4}}^{T-\frac{\bar a_2}{8}}\!\! \int_0^{\delta-\bar a_1}\!\!\int_0^1 \frac{1}{k_2} y^2dxdadt.
\]
Clearly, as in \eqref{secondanew},
\begin{equation}\label{secondanew'}
\begin{aligned}
&\int_{T-\frac{\bar a_2}{4}}^{T-\frac{\bar a_2}{8}}\!\! \int_0^{\delta-\bar a_1}\!\!\int_0^1 \frac{1}{k_2} y^2dxdadt \le
\int_{T-\frac{\bar a_2}{2}}^{T-\frac{\bar a_2}{8}}\int_0^{\delta - \bar a_2} \int_0^1  \frac{1}{k_2}y^2(t,a,x) dxdadt\\
&\le C\int_{T-\frac{\bar a_2}{2}}^{T-\frac{\bar a_2}{8}}\int_0^{\delta - \bar a_2}\int_0^1 \frac{1}{k_2}y^2_T(T+a-t,x)dxdadt \\
&+C\int_{T-\frac{\bar a_2}{2}}^{T-\frac{\bar a_2}{8}}\int_0^{\delta - \bar a_2}\int_0^1  \left(\int_a^{T+a-t}\frac{1}{k_2}y^2(s+t-a,0,x)ds \right)dxdadt\\
& \le C\int_{\frac{\bar a_2}{8}}^{\delta-\frac{\bar a_2}{2}} \int_0^1 \frac{1}{k_2}y^2_T(\sigma,x)dxd\sigma dz \\
&+C\int_{T-\frac{\bar a_2}{2}}^{T-\frac{\bar a_2}{8}}\int_0^{\delta - \bar a_2}\int_0^1  \left(\int_0^{T-t}\frac{1}{k_2}y_T^2(\sigma,x)d\sigma \right)dxdadt\\
& \le C\int_0^\delta\int_0^1\frac{1}{k_2}y_T^2(\sigma,x)d\sigma dx.
\end{aligned}
\end{equation}
Now, consider  $\ds \int_{T-\frac{\bar a_2}{4}}^{T-\frac{\bar a_2}{8}} \int_{\delta- \bar a_1}^A\int_0^1\frac{1}{k_1} z^2dxdadt$. As before,  by Theorem \ref{Cor2} applied to $\bar Q_1:=\left(T-\ds \frac{\bar a_2}{2}, T- \ds\frac{\bar a_2}{10}\right) \times (0, A)\times (0,1)$, replacing  $\tilde \varphi_2$,  $\tilde \Theta_2$ by $\tilde \varphi_1$, $\tilde \Theta_1$, respectively, where  $\tilde \varphi_1$  is the function  $\varphi_1$ associated to $\tilde \Theta_1$  according to \eqref{13}, and $\tilde \Theta_1$  is defined in \eqref{thetatilde} with $T_\gamma:= T-\ds\frac{ \bar a_2}{2}$, $T_\beta:=T- \ds\frac{\bar a_2}{10}$, $\gamma =0$, one has
\[
\int_{T-\frac{\bar a_2}{4}}^{T-\frac{\bar a_2}{8}}  \int_{\delta-\bar a_1}^A\int_0^1 \frac{1}{k_1} z^2(t,a,x) dxdadt \le
C\left(\int_{\bar Q_1}\frac{f^{2}}{k_1}dxdadt+ \int_0^T \int_0^A\int_ \omega \frac{z^2}{k_1} dx dadt\right),
\]
where again $f(t,a,x):=\mu_{21}(t, a, x)y - \beta_1(a,x)z(t,0,x)$. Hence, proceeding as in the previous theorem and as in \eqref{secondanew'}, one has
\[
\begin{aligned}
&\int_{T-\frac{\bar a_2}{4}}^{T-\frac{\bar a_2}{8}} \int_{\delta-\bar a_1}^A\int_0^1 \frac{1}{k_1} z^2(t,a,x) dxdadt \\
&\le  C\left(\int_{\bar Q_1}\frac{y^2}{k_2}dxdadt + \int_0^{\bar a_1} \int_0^1\frac{z_T^2(a,x) }{k_1}dxda + \int_{T-\frac{\bar a_2}{2}}^{T-\frac{\bar a_2}{10}}\int_0^1 \left(\int_0^{T-t} \frac{1}{k_1}y^2(s+t,s,x)ds\right)dxdt
\right)\\
&+ C\int_0^T \int_0^A\int_ \omega \frac{z^2}{k_1} dx dadt\\
&\le C \int_{T-\frac{\bar a_2}{2}}^{T-\frac{\bar a_2}{10}}\int_0^\delta \int_0^1\frac{y^2}{k_2}dxdadt
+C\int_0^{\bar a_2} \int_0^1\frac{y_T^2(a,x) }{k_2}dxda + \int_0^T \int_0^A\!\!\int_ \omega \frac{y^2}{k_2}  dx dadt\\
&+ C\int_0^{\bar a_1} \int_0^1\frac{z_T^2(a,x) }{k_1}dxda + C\int_{T-\frac{\bar a_2}{2}}^{T-\frac{\bar a_2}{10}}\int_0^1 \left(\int_0^{T-t} \frac{1}{k_1}y^2(s+t,s,x)ds\right)dxdt+ C\int_0^T \int_0^A\int_ \omega \frac{z^2}{k_1} dx dadt\\
&\le  C\int_0^\delta \int_0^1\frac{y_T^2(a,x)}{k_2}dxda
+ \int_0^T \int_0^A\!\!\int_ \omega \frac{y^2}{k_2}  dx dadt+ C\int_0^\delta \int_0^1\frac{z_T^2(a,x) }{k_1}dxda \\&+ C\int_{T-\frac{\bar a_2}{2}}^{T-\frac{\bar a_2}{10}}\int_0^1 \left(\int_0^{T-t} \frac{1}{k_1}y^2(s+t,s,x)ds\right)dxdt+ C\int_0^T \int_0^A\int_ \omega \frac{z^2}{k_1} dx dadt.
\end{aligned}
\]
Now, it remains to estimate
\[
\begin{aligned}
 \int_{T-\frac{\bar a_2}{2}}^{T-\frac{\bar a_2}{10}}\int_0^1 \left(\int_0^{T-t} \frac{1}{k_1}y^2(s+t,s,x)ds\right)dxdt.
\end{aligned}
\]
As in \eqref{startondino}, one can prove
\[
\begin{aligned}
 &\int_{T-\bar a_2}^{T-\frac{\bar a_2}{10}}\int_0^1 \left(\int_0^{T-t} \frac{1}{k_1}y^2(s+t,s,x)ds\right)dxdt \le C \int_0^{\bar a_2}\int_0^1\frac{y_T^2(\tau,x)}{k_2} dxd\tau.
\end{aligned}
\]
Hence, \eqref{ribo1} holds.

\end{proof}
\begin{Remark}\label{regolarita}
Using a density argument one can prove \eqref{ribo} and \eqref{ribo1} for every
solution $(z,y)\in C([0,T]; L^2(0,A; 
\mathbb{L})) \cap L^2 (0,T; H^1(0,A;\mathbb{H}))$ of \eqref{adjoint}.
\end{Remark}

\subsection{Carleman inequalities and observability inequalities when the degeneracy is at $1$.}\label{sec-3-2}
In this case, in order to obtain $\omega-$local Carleman estimates and $\omega-$local observability inequality, in place of $\varphi_i$, we consider the weight functions
\begin{equation}\label{funzioninew}
\bar\varphi_i(t,a,x):=\Theta(t,a)(\bar p_i(x) - 2 \|\bar p_i\|_{L^\infty(0,1)}),
\end{equation}
where $\Theta$ is as in \eqref{571} and
$\displaystyle \bar p_i(x):=\int_0^x\frac{y-1}{k_i(y)}~e^{R(y-1)^2}dy$, with $R>0$ and $i=1,2$. Assume
\begin{Assumptions}\label{AssP1} The functions
$k_i\in C^0[0,1]\bigcap C^2[0,1)$  are such that $k_i(1)=0$, $k_i>0$ on
$[0,1)$ and 
\begin{enumerate}
\item $k_i$ satisfies Hypothesis \ref{BAss02}.1, for $i=1,2$,
%\item there exist $\varepsilon_i\in (0,1]$ and $M_i\in (0,2)$ such that
%the functions $\displaystyle\frac{(x-1)k_{i,x}}{k_i(x)}$\normalsize\:
%$\in L^{\infty}(1-\varepsilon_i,1)$,
%$\displaystyle\frac {(x-1)k_{i,x}(x)}{k_i(x)} \le M_i $, for all $x \in (0,1]$
%and $\displaystyle\left( \frac{(x-1)k_{i,x}(x)}{k_i(x)}\right)_{x}\in L^{\infty}(1-\varepsilon_i,1)$ $i=1,2$.
\item there exists a positive constant $C$ such that $k_1(x)\ge k_2(x)$ for all $x \in [0,1]$.
\end{enumerate}
\end{Assumptions}
Proceeding as in the proof of Theorem \ref{Carleman}, one can prove the next estimate:
\begin{Theorem}\label{Carleman1}
Assume Hypotheses  $\ref{ratesAss}$, $\ref{Assw}$  and $\ref{AssP1}$. Take  $f \in L^2_{\frac{1}{k_1}}(Q)$ and $g\in L^2_{\frac{1}{k_2}}(Q)$. Then,
there exist two strictly positive constants $C$ and $s_0$ such that every
solution $(z,y)\in L^2\big(Q_{T,A}; \mathbb{K}\big) \cap H^1\big(0, T; H^1(0,A;\mathbb{H})\big)$ of \eqref{adjoint_f}
satisfies, for all $s \ge s_0$,
\[
\begin{aligned}
&\int_{Q}\left(s \Theta z_x^2
                + s^3\Theta^3\text{\small$\displaystyle\Big(\frac{x-1}{k_1}\Big)^2$\normalsize}
                  z^2\right)e^{2s\varphi_1}dxdadt  
 \le  C \int_{Q}(f^2+ y^2)\frac{1}{k_1}dxdadt +C\int_{Q_{T,A}}\int_ \omega \frac{z^2}{k_1} dx dadt
\end{aligned}\]
and
\[
\begin{aligned}
\int_{Q}\left(s \Theta y_x^2
                + s^3\Theta^3\Big(\frac{x-1}{k_2}\Big)^2
                  y^2\right)e^{2s\varphi_2}dxdadt
&\le
C\left(\int_{Q}\text{\small$\frac{h^{2}}{k_2}$\normalsize}~dxdadt+\int_{Q_{T,A}}\int_ \omega \frac{y^2}{k_2} dx dadt\right).
\end{aligned}\]
\end{Theorem}
As a consequence of the previous result, it holds:
\begin{Theorem}\label{Theorem4.4'} Assume Hypotheses $\ref{ratesAss}$, $\ref{Assw}$, $\ref{AssP1}$ and  \ref{conditionbeta}. Then, for every $\delta \in (0,A)$, 
there exists a strictly positive constant $C=C(\delta)$ such that every
solution $(z,y)\in L^2\big(Q_{T,A}; \mathbb{K}\big) \cap H^1\big(0, T; H^1(0,A;\mathbb{H})\big)$ of \eqref{adjoint}
satisfies
\eqref{stima1}
and
\eqref{stima2}.
Moreover, if $y_T(a,x)=0$ for all $(a,x) \in (0,  \bar a_2) \times (0,1)$ and  $z_T(a,x)=0$ for all $(a,x) \in (0,  \bar a_1) \times (0,1)$, one has
\[
\begin{aligned}
&\int_{Q_{A,1}} \frac{1}{k_2} y^2(T-\bar a_2, a,x) dxda \le C\int_{Q_{T,A}}\!\!\int_ \omega\frac{1}{k_2} y^2  dx dadt+ C\int_0^T \int_0^\delta\!\!\int_ 0^1 \frac{1}{k_2} y^2 dx dadt
\end{aligned}
\]
and
\[\begin{aligned}
 \int_{Q_{A,1}} \frac{z^2}{k_1} (T-\bar a_1,a,x) dxda &\le C\left(\int_{Q_{T,A}}\!\!\int_ \omega\left( \frac{1}{k_1} z^2 + \frac{1}{k_2} y^2\right)dx dadt+ \int_0^T \int_0^\delta\!\!\int_ 0^1\left( \frac{1}{k_1} z^2 + \frac{1}{k_2} y^2\right)dx dadt\right).
\end{aligned}\]
\end{Theorem}
Moreover,
\begin{Theorem}\label{CorOb1''}Assume Hypotheses $\ref{ratesAss}$, $\ref{Assw}$, $\ref{AssP1}$ and  \ref{conditionbeta}. Then, for every $\delta \in (\gamma,A)$, where $\gamma = \max \{\bar a_1, \bar a_2\}$,
there exists a strictly positive constant $C=C(\delta)$ such that every
solution $(z,y)\in L^2\big(Q_{T,A}; \mathbb{K}\big) \cap H^1\big(0, T; H^1(0,A;\mathbb{H})\big)$ of \eqref{adjoint}
satisfies
\eqref{ribo}
and
\eqref{ribo1}.
\end{Theorem}
Again, using a density argument one can prove \eqref{ribo} and \eqref{ribo1} even  if  $(z,y)\in C([0,T]; L^2(0,A; 
\mathbb{L})) \cap L^2 (0,T; H^1(0,A;\mathbb{H}))$; hence, Remark \ref{regolarita} still holds  if $k_i(1)=0$, for $i=1,2$.
\section{Null controllability via observability inequality}\label{section4}
In this section we will deduce a null controllability result for \eqref{1} by
Theorems \ref{CorOb1'}, \ref{CorOb1''} and Remark \ref{regolarita}. Actually, thanks to Theorems  \ref{CorOb1'} and \ref{CorOb1''}, we can obtain a null controllability result for the following intermediate problem:
\begin{equation} \label{1_medio}
\begin{cases}
\ds\frac{\partial u}{\partial t}+\frac{\partial u}{\partial a}- k_1(x)u_{xx} + \mu_{11}(t,x,a) u =g(t,x,a) \chi_{\omega}, & (t,a,x)\in \tilde Q_1, 
\\
\ds\frac{\partial v}{\partial t}+\frac{\partial v}{\partial a}-k_1(x)v_{xx} + \mu_{22}(t,x,a) v-\mu_{21}(t,x,a) u=0, & (t,a,x)\in \tilde Q_2,
\\
u(t, a, 0) = u(t,a, 1) = 0, & (t,a)\in \tilde Q_{1,T,A},\\
 v(t,a, 0) = v(t, a, 1) =0,  & (t,a)\in \tilde Q_{2,T,A},\\
u(T-\bar a_1, a, x) =u_0(a,x),\quad v(T-\bar a_2,a, x) = v_0(a,x), &  (a,x) \in Q_{A,1},\\
u(t, 0, x)=\int_0^A \beta_1 (a, x)u (t, a, x) da,  & (t,x) \in \tilde Q_{1,T,1},\\
v(t, 0, x)=\int_0^A \beta_2 (a, x)v(t, a, x) da,  & (t,x) \in \tilde Q_{2, T,1},
\end{cases}
\end{equation}
where, we recall, $\tilde Q_i = (T-\bar a_i, T)\times Q_{A,1}$, $ \tilde Q_{i,T,1}=  (T-\bar a_i, T) \times (0,1)$ and $ \tilde Q_{i, T,A}=  (T-\bar a_i, T) \times (0,A)$, $i=1,2$.
In particular, the following result  holds:
\begin{Theorem}\label{ultimo}Assume Hypotheses $\ref{ratesAss}$, $\ref{Assw}$,  \ref{conditionbeta} and $\ref{AssP}$ or $\ref{AssP1}$ and  suppose $\bar a_1= \bar a_2$.  Fix
 $T >0$ and $(u_0, v_0) \in L^2(0,A;\mathbb{L} )$. Then
 for every $\delta \in (\gamma,A)$,
 there exist a contros  $g_\delta \in L^2_{\frac{1}{k_1}}(\tilde Q_1)$ such that the solution $(u_\delta, v_\delta)\in C([T-\bar a_1,T]; L^2(0,A; 
\mathbb{L})) \cap L^2 (T-\bar a_1,T; H^1(0,A;\mathbb{H}))$ of 
\eqref{1_medio}
satisfies
\begin{equation}\label{chissa}
u_\delta(T,a,x) =v_\delta(T,a,x) =0 \quad \text{a.e. } (a,x) \in (\delta, A) \times (0,1).
\end{equation}
Moreover, there exists $C=C(\delta)>0$, such that
\begin{equation}\label{stimah}
\|g_\delta\|_{L^2_{\frac{1}{k_1}}(\tilde Q_1)} \le  C\left(\|v_0\|_{L^2_{\frac{1}{k_2}}(Q_{A,1})}+  \|u_0\|_{L^2_{\frac{1}{k_1}}(Q_{A,1})}\right).
\end{equation}
\end{Theorem}
Before proving the previous result observe that, if $\bar a_1= \bar a_2$, then
$(u,v)$ solves \eqref{1_medio} if and only if $v$ satisfies 
\begin{equation} \label{Pv}
\begin{cases}
\ds\frac{\partial v}{\partial t}+\frac{\partial v}{\partial a}-k_2(x)v_{xx} + \mu_{22}(t,x,a) v-\mu_{21}(t,x,a) u=0, & (t,a,x)\in \tilde Q,
\\
v(t,a, 0) = v(t, a, 1) = 0, & (t,a)\in \tilde Q_{T,A},\\
v(T-\bar a_1,a, x) = v_0(a,x), &  (a,x) \in Q_{A,1},\\
v(t, 0, x)=\int_0^A \beta_2 (a, x)v(t, a, x) da,  & (t,x) \in \tilde Q_{T,1}
\end{cases}
\end{equation}
where $u$ solves
\begin{equation} \label{Pu}
\begin{cases}
\ds\frac{\partial u}{\partial t}+\frac{\partial u}{\partial a}- k_1(x)u_{xx} + \mu_{11}(t,x,a) u =g(t,x,a) \chi_{\omega}, & (t,a,x)\in \tilde Q, 
\\
u(t, a, 0) = u(t,a, 1) = 0, & (t,a)\in \tilde Q_{T,A},\\
u(T-\bar a_1, a, x) =u_0(a,x), &  (a,x) \in Q_{A,1},\\
u(t, 0, x)=\int_0^A \beta_1 (a, x)u (t, a, x) da,  & (t,x) \in \tilde Q_{T,1},
\end{cases}
\end{equation}
where $\tilde Q:= (T-\bar a_1, T) \times Q_{A,1}$,  $\tilde Q_{T,A}= (T-\bar a_1, T) \times (0,A)$ and $\tilde Q_{T,1}:= (T-\bar a_1, T)\times (0,1)$.

\begin{proof}[Proof of Theorem \ref{ultimo}]
Take $(g_1, g_2) \in L^2(0, A; \mathbb{L})$ such that $g_i(A,x)=0$ in $(0,1)$, $i=1,2$, and fix $\delta \in (\gamma, A)$.
Let $y$ be the solution of
\begin{equation}\label{ad_2'}
\begin{cases}
\ds \frac{\partial y}{\partial t} + \frac{\partial y}{\partial a}
+k_2(x)y_{xx}-\mu_{22}(t, a, x)y =-\beta_2(a,x)y(t,0,x),& (t,a,x) \in \tilde Q,
\\
 y(t, a, 0)=y(t, a, 1)=0, & (t,a) \in \tilde Q_{T,A},\\
y(t,A,x)=0, & (t,x) \in \tilde Q_{T,1},\\
y(T,a,x)= y_T(a,x)):= 
g_2(a,x)\chi_{(\delta, A)}, & (a,x) \in Q_{A,1}.
\end{cases}
\end{equation}
Now, fixed $v_0 \in L^2(0,A; L^2_{\frac{1}{k_2}}(0,1))$, define
\[
\begin{aligned}
J(g_2) =\frac{1}{2}\int_{T_2}^T \int_0^A\int_\omega \frac{y^2}{k_2}  dxda dt +  \int_0^A\int_0^1 \frac{1}{k_2}y(T-\bar a_2,a, x)v_0(a,x) dxda.
\end{aligned}
\]
The functional $J$ is  strictly convex, continuous and coercive over the Hilbert space $\mathcal H$ defined by the completion of $L^2(\delta, A; L^2_{\frac{1}{k_2}}(0,1))$  with respect to the norms $\|v\|_{L^2(\tilde Q_{T,A}\times \omega)}$ (see \cite{Ain3}). Thus, there exists a unique minimum, $\hat g_2$, of $J$ and $\hat g_2(A,x)=0$ in $(0,1)$. Let $\hat y$  be the solution  of\eqref{ad_2'} associated to $\hat g_2$.

Now, fixed  $u_0 \in L^2(0,A; L^2_{\frac{1}{k_1}}(0,1))$,
 let $z$ be the solution of 
\begin{equation}\label{ad}
\begin{cases}
\ds \frac{\partial z}{\partial t} + \frac{\partial z}{\partial a}
+k_1(x)z_{xx}-\mu_{11}(t, a, x)z =-\mu_{21}(t, a, x)\hat y - \beta_1(a,x)z(t,0,x),& (t,a,x) \in \tilde Q,  \\
  z(t, a, 0)=z(t, a, 1)=0, & (t,a) \in \tilde Q_{T,A},\\
  z(t,A,x)=0, & (t,x) \in \tilde Q_{T,1},\\
z(T,a,x)= z_T(a,x)):= 
g_1(a,x)\chi_{(\delta, A)}, & (a,x) \in Q_{A,1}.
\end{cases}
\end{equation}
and
 define
\[
\begin{aligned}
F(g_1) &=\frac{1}{2}\int_{T_1}^T \int_0^A\int_\omega \frac{z^2}{k_1}  dxda dt+  \int_0^A\int_0^1 \frac{1}{k_1}z(T-\bar a_1,a, x)u_0(a,x) dxda.
\end{aligned}
\]
As before, there exists a unique minimum, $\hat g_1$, of $F$ and $\hat g_1(A,x)=0$ in $(0,1)$. Let $\hat z$  be the solution  of\eqref{ad} associated to $\hat g_1$.

 Since $\hat g_2$ is the minimum of $J$, it results
\begin{equation}\label{*2}
\begin{aligned}
0=\left[ \frac{d}{dt} J(\hat g_2 + t g_2)\right]_{t=0} = \int_{T-\bar a_1}^T\int_0^A \int_\omega \frac{1}{k_2} y\hat y  dxdadt  + \int_0^A\int_0^1 \frac{1}{k_2} y({T-\bar a_1},a, x) v_0(a, x)  dxda,
\end{aligned}
\end{equation}
for all $g_2 \in L^2(Q_{A,1})$  such that $g_2(A,x)=0$ in $(0,1)$. 
Analogously, we have
\begin{equation}\label{a2''}
\begin{aligned}
0=\left[ \frac{d}{dt} F(\hat g_1 + t g_1)\right]_{t=0} 
& =\int_{T-\bar a_1}^T\int_0^A \int_\omega \frac{1}{k_1} z\hat z  dxdadt+ \int_0^A\int_0^1 \frac{1}{k_1}z({T-\bar a_1},a, x) u_0(a, x)  dxda,
\end{aligned}
\end{equation}
for all $g_1 \in L^2(Q_{A,1})$  such that $g_1(A,x)=0$ in $(0,1)$.
In particular, for $g_2 = \hat g_2$ and $g_1= \hat g_1$, it results
\begin{equation}\label{a2}
\int_{T-\bar a_1}^T \int_0^A\int_\omega  \frac{\hat y^2}{k_2} dxda dt =- \int_0^A\int_0^1\frac{1}{k_2} \hat y(T-\bar a_1,a,x) v_0(a,x)  dxda,
\end{equation}
and
\begin{equation}\label{a2'}
\int_{T-\bar a_1}^T \int_0^A\int_\omega  \frac{\hat z^2}{k_1} dxda dt   =- \int_0^A\int_0^1\frac{1}{k_1} \hat z(T-\bar a_1,a,x) u_0(a,x)  dxda.
\end{equation}
By H\"older's inequality, Remark \ref{regolarita}, Theorems \ref{CorOb1'} and \ref{CorOb1''} applied to $\hat y$ in $\tilde Q_2$ and  using the fact that $y_T(a,x)=0$ for all $(a,x) \in (0, \delta)\times (0,1)$, one has
\begin{equation}\label{a3}
\begin{aligned}
&\left| \int_0^A\int_0^1\frac{1}{k_2} \hat y({T-\bar a_1},a, x) v_0(a, x)  dxda \right| \le \left(\int_0^A\int_0^1\frac{1}{k_2} \hat y^2({ T-\bar a_1},a, x)  dxda\right)^{\frac{1}{2}} \left(\int_0^A\int_0^1 \frac{1}{k_2}v_0^2(a,x) dxda\right)^{\frac{1}{2}}\\
& \le C\left(\int_{T-\bar a_1}^T \int_0^A\int_ \omega \frac{\hat y^2}{k_2} dx dadt\right)^{\frac{1}{2}} \|v_0\|_{L^2_{\frac{1}{k_2}}(Q_{A,1})}.
\end{aligned}
\end{equation}
Thus, by \eqref{a2} and \eqref{a3}, it follows
\begin{equation}\label{a4}
\begin{aligned}
 \int_{T-\bar a_1}^T\int_0^A\io \frac{1}{k_2}\hat y^2(t,a, x) dxda dt\le  C \left(\int_{T-\bar a_1}^T \int_0^A\int_ \omega \frac{\hat y^2}{k_2} dx dadt\right)^{\frac{1}{2}}\|v_0\|_{L^2_{\frac{1}{k_2}}(Q_{A,1})}.
\end{aligned}
\end{equation}
Hence
\begin{equation}\label{stimay}
\left( \int_{T-\bar a_1}^T\int_0^A\io \frac{\hat y^2}{k_2}dxda dt\right)^{\frac{1}{2}} \le C\|v_0\|_{L^2_{\frac{1}{k_2}}(Q_{A,1})}.
\end{equation}
Analogously, by H\"older's inequality, Remark \ref{regolarita}, Theorems \ref{CorOb1'} and \ref{CorOb1''} applied to $\hat z$ in $\tilde Q$ and using the fact that $z_T(a,x) =0$ for all $(a,x) \in (0, \delta)\times (0,1)$, one has
\begin{equation}\label{a3'}
\begin{aligned}
&\left| \int_0^A\int_0^1\frac{1}{k_1} \hat z({T-\bar a_1},a, x) u_0(a, x)  dxda \right| \le \left(\int_0^A\int_0^1 \frac{1}{k_1}u_0^2(a,x) dxda\right)^{\frac{1}{2}} \left(\int_0^A\int_0^1\frac{1}{k_1} \hat z^2({ T-\bar a_1},a, x)  dxda\right)^{\frac{1}{2}}\\
& \le C \|u_0\|_{L^2_{\frac{1}{k_1}}(Q_{A,1})}\left(\int_{T-\bar a_1}^T \int_0^A\int_ \omega \frac{\hat z^2}{k_1} dx dadt+\int_{T-\bar a_1}^T \int_0^A\int_ \omega \frac{\hat y^2}{k_2} dx dadt\right)^{\frac{1}{2}}\\
&\le 2C\|u_0\|_{L^2_{\frac{1}{k_1}}(Q_{A,1})}^2+\frac{1}{2}\left(\int_{T-\bar a_1}^T \int_0^A\int_ \omega \frac{\hat z^2}{k_1} dx dadt+\int_{T-\bar a_1}^T \int_0^A\int_ \omega \frac{\hat y^2}{k_2} dx dadt\right).
\end{aligned}
\end{equation}
In the last inequality we have used  the Cauchy's inequality $\ds ab \le \epsilon a^2 + \frac{1}{\epsilon}b^2$ with $\ds \epsilon =\frac{1}{2}$.
Thus, by \eqref{a2'}, \eqref{stimay} and \eqref{a3'}
%\begin{equation}\label{a4'}
\[
\begin{aligned}
\frac{1}{2} \int_{T-\bar a_1}^T \int_0^A\int_ \omega \frac{\hat z^2}{k_1} dx dadt& \le \frac{1}{2} \int_{T-\bar a_1}^T \int_0^A\int_ \omega \frac{\hat y^2}{k_2} dx dadt+   2C \|u_0\|_{L^2_{\frac{1}{k_1}}(Q_{A,1})} ^2\\
&\le C\left(\|v_0\|_{L^2_{\frac{1}{k_2}}(Q_{A,1})}^2+  \|u_0\|_{L^2_{\frac{1}{k_1}}(Q_{A,1})}^2\right).
\end{aligned}
\]
Hence
\begin{equation}\label{stimaz}
\left( \int_{T-\bar a_1}^T\int_0^A\io \frac{\hat z^2}{k_1}dxda dtt\right)^{\frac{1}{2}}\le C\left(\|v_0\|_{L^2_{\frac{1}{k_2}}(Q_{A,1})}+  \|u_0\|_{L^2_{\frac{1}{k_1}}(Q_{A,1})}\right).
\end{equation}
Now, define 
$
g:=\hat z\chi_\omega 
$
and let $u$ be the solution of \eqref{Pu} in $\tilde Q_1$ associated to $g$ and $u_0$ and  let $v$ be the solution of \eqref{Pv} in $\tilde Q_2$ associated to $u$ and $v_0$. By \eqref{stimaz}, it follows
\[
\|g\|_{L^2_{\frac{1}{k_1}}(\tilde Q_1)} = \left( \int_{T-\bar a_1}^T\int_0^A\io \frac{\hat z^2}{k_2}dxda dt\right)^{\frac{1}{2}}\le C\left(\|v_0\|_{L^2_{\frac{1}{k_2}}(Q_{A,1})}+  \|u_0\|_{L^2_{\frac{1}{k_1}}(Q_{A,1})}\right).
\]
Moreover, multiplying the equation of \eqref{ad} by $\ds\frac{u}{k_1}$ and integrating over $\tilde Q_1$, one has:
%0& = \int_{\tilde Q_1} \left(\ds \frac{\partial z}{\partial t} + \frac{\partial z}{\partial a}
%+k_1(x)z_{xx}-\mu_{11}(t, a, x)z +\mu_{21}%(t, a, x)\hat y + \beta_1(a,x)z(t,0,x)%\right)  \frac{u}{k_1}dxda dt %\Longleftrightarrow \\
\[
\begin{aligned}
0& = \int_\delta^A \int_0^1\frac{1}{k_1}u(T,a,x)g_1(a,x)\  dxda -\int_{Q_{A,1}}\frac{1}{k_1} u_0 z(T-\bar a_1,a,x)  dxda - \int_{T-\bar a_1}^T\int_0^1\frac{1}{k_1} z(t,0,x)u(t,0,x)  dxdt\\
& +\int_{\tilde Q_1} \frac{1}{k_1}  \beta_1(a,x) z(t,0,x) u(t,a,x) dxdadt+ \int_{\tilde Q_1}  \mu_{21}(t,a,x) \frac{\hat yu}{k_1} dxdadt\\
&- \int_{\tilde Q_1}\frac{z}{k_1}\left(\frac{\partial u}{\partial t} + \frac{\partial u}{\partial a}
-k_1(x)u_{xx}+\mu_{11}(t, a, x)u\right)   dxda dt.
\end{aligned}
\]
Recall that $u(t, 0, x)=\int_0^A \beta_1 (a, x)u (t, a, x) da$ and
 $\ds\frac{\partial u}{\partial t} + \frac{\partial u}{\partial a}
-k_1(x)u_{xx}+\mu_{11}(t, a, x)u= g\chi_\omega $; hence
\[
\begin{aligned}
0&=\int_\delta^A \int_0^1\frac{1}{k_1}u(T,a,x)g_1(a,x)\  dxda -\int_{Q_{A,1}}\frac{1}{k_1} u_0 z(T-\bar a_1,a,x)  dxda\\
&- \int_{T-\bar a_1}^T\int_0^A \int_\omega\frac{z\hat z}{k_1}  dxda dt + \int_{\tilde Q_1} \mu_{21}\frac{\hat yu}{k_1} dxdadt.
\end{aligned}
\]
Thus, being by \eqref{a2''},
\[
\int_{T- \bar a_1}^T\int_0^A \int_\omega \frac{1}{k_1} z \hat z  dxdadt=- \int_0^A\int_0^1 \frac{1}{k_1}z(T-\bar a_1,a, x) u_0(a, x)  dxda,
\]
it follows
\begin{equation}\label{EQu}
\int_\delta^A \int_0^1\frac{1}{k_1}u(T,a,x)g_1(a,x)\  dxda= -\int_{\tilde Q_1} \mu_{21} \frac{u \hat y}{k_1}dxdadt
\end{equation}
for all $g_1\in L^2_{\frac{1}{k_1}}(Q_{A,1})$   with $g_1(A,x)=0$ in $(0,1)$, $i=1,2$. In particular, for $g_1 \equiv0$, it results
\[
\int_{\tilde Q_1} \mu_{21} \frac{u \hat y}{k_1}dxdadt=0.
\]
Hence, by \eqref{EQu},
\begin{equation}\label{EQunew}
\int_\delta^A \int_0^1\frac{1}{k_1}u(T,a,x)g_1(a,x)\  dxda= 0
\end{equation}
for all $g_1\in L^2_{\frac{1}{k_1}}(Q_{A,1})$   with $g_1(A,x)=0$ in $(0,1)$, $i=1,2$. Thus,
\[
u(T,a,x)=0 \quad \text{a.e.}  (a,x) \in (\delta,A) \times (0,1).
\]
 Now, multiplying the equation of \eqref{ad_2'} by $\ds\frac{v}{k_2}$ and integrating over $\tilde Q_2$, one has:
%0& = \int_{\tilde Q_2} \left(\frac{\partial y}{\partial t} + \frac{\partial y}{\partial a}
%+k_2(x) y_{xx}-\mu_{22}(t, a, x) y +\beta_2(a,x) y(t,0,x)\right)  \frac{v}{k_2}%dxda dt \Longleftrightarrow \\

\[
\begin{aligned}
0& = \int_\delta^A \int_0^1\frac{1}{k_2}v(T,a,x) g_2(a,x)\  dxda -\int_{Q_{A,1}}\frac{1}{k_2} v_0  y(T-\bar a_1,a,x)  dxda \\
&- \int_{T-\bar a_1}^T\int_0^1\frac{1}{k_2} y(t,0,x) v(t,0,x)  dxdt+\int_{\tilde Q_2} \frac{1}{k_2}  \beta_2(a,x) \hat y(t,0,x) v(t,a,x) dxdadt\\
&- \int_{\tilde Q_2}\frac{ y}{k_2}\left(\frac{\partial v}{\partial t} + \frac{\partial v}{\partial a}
-k_2(x)v_{xx}+\mu_{22}(t, a, x)v\right)   dxda dt.
\end{aligned}
\]
Recall that $v(t, 0, x)=\int_0^A \beta_2 (a, x)v (t, a, x) da$ and
 $\ds\frac{\partial v}{\partial t} + \frac{\partial v}{\partial a}
-k_2(x)v_{xx}+\mu_{22}(t, a, x)v= \mu_{21}u$; hence
\[
\begin{aligned}
0&=\int_\delta^A \int_0^1\frac{1}{k_2}v(T,a,x)\ g_2(a,x)\  dxda -\int_{Q_{A,1}}\frac{1}{k_2} v_0  y(T-\bar a_1,a,x)  dxda- \int_{\tilde Q_2} \mu_{21}\frac{u y}{k_2} dxdadt.
\end{aligned}
\]
Thus, being by \eqref{*2},
it follows
\begin{equation}\label{EQv}
\int_\delta^A \int_0^1\frac{1}{k_2}v(T,a,x) g_2(a,x)\  dxda=\int_{\tilde Q_2}\mu_{21} \frac{u y}{k_2}dxdadt
\end{equation}
for all $g_2 \in L^2_{\frac{1}{k_2}}(Q_{A,1})$   with $g_2(A,x)=0$ in $(0,1)$. In particular, for $g_2 \equiv0$, it follows
\[
\int_{\tilde Q_2}\mu_{21} \frac{u y}{k_2}dxdadt=0.
\]
Proceeding as before,  we can conclude that 
\[
v(T,a,x)=0 \quad \text{a.e.} \;  (a,x) \in (\delta,A) \times (0,1).\]
Hence, the thesis follows.
\end{proof}

\begin{Theorem}\label{ultimonew}Assume Hypotheses $\ref{ratesAss}$, $\ref{Assw}$,  \ref{conditionbeta} and $\ref{AssP}$ or $\ref{AssP1}$, suppose $k_1(x)=k_2(x)$ for all $x \in (0,1)$ and $\bar a_1=\bar a_2$. Fix  $T >0$ and $(u_0, v_0) \in L^2(0,A;\mathbb{L} )$. Then
 for every $\delta \in (\gamma,A)$, 
 there exist a control  $g_\delta \in L^2_{\frac{1}{k_1}}(Q)$ such that the solution $(u_\delta, v_\delta)\in C([0,T]; L^2(0,A; 
\mathbb{L})) \cap L^2 (0,T; H^1(0,A;\mathbb{H}))$ of   \eqref{1}
satisfies
\[
u_\delta(T,a,x) =v_\delta(T,a,x) =0 \quad \text{a.e. } (a,x) \in (\delta, A) \times (0,1).
\]
Moreover, there exists $C=C(\delta)>0$  such that
\begin{equation}\label{stimah1}
\|g_\delta\|_{L^2_{\frac{1}{k_1}}(Q)} \le C \left(\|v_0\|_{L^2_{\frac{1}{k_1}}(Q_{A,1})} + C \|u_0\|_{L^2_{\frac{1}{k_1}}(Q_{A,1})}\right).
\end{equation}
\end{Theorem}
\begin{proof} 
Set $T_1= T- \bar a_1$ (clearly $T_2=T_1$). By Theorem  \ref{esistenza}, there exists a unique solution $(U,V)$  of
\begin{equation} \label{1new2}
\begin{cases}
\ds\frac{\partial U}{\partial t}+\frac{\partial U}{\partial a}- k_1(x)U_{xx} + \mu_{11}(t,x,a) U =0, & (t,a,x)\in \bar Q, 
\\[5pt]
\ds\frac{\partial V}{\partial t}+\frac{\partial V}{\partial a}-k_1(x)V_{xx} + \mu_{22}(t,x,a) V-\mu_{21}(t,x,a) U=0, & (t,a,x)\in \bar Q,
\\
U(t, a, 0) = U(t,a, 1)  = 0, & (t,a)\in\bar  Q_{T_1,A},\\
V(t,a, 0) = V(t, a, 1)= 0, & (t,a)\in\bar  Q_{T_1,A},\\
U(0, a, x) =u_0(a,x),\quad V(0,a, x) = v_0(a,x), &  (a,x) \in Q_{A,1},\\
U(t, 0, x)=\int_0^A \beta_1 (a, x)U (t, a, x) da,  & (t,x) \in \bar Q_{T_1,1},\\
V(t, 0, x)=\int_0^A \beta_2 (a, x)V(t, a, x) da,  & (t,x) \in \bar Q_{T_1,1},
\end{cases}
\end{equation}
where $\bar Q = (0, T_1)\times Q_{A,1}$, $ \bar Q_{T_1,1}=  (0,T-\bar a_1) \times (0,1)$ and $ \bar Q_{T_1,A}=  (0,T-\bar a_1) \times (0,A)$, $i=1,2$.

Set $\tilde u_0(a,x) :=U(T_1, a, x)$, $\tilde v_0(a,x) :=V(T_1, a, x)$; clearly $\tilde u_0, \tilde v_0 \in L^2_{\frac{1}{k_1}}(Q_{A,1})$.
Now, consider 
\begin{equation} \label{1new1}
\begin{cases}
\ds\frac{\partial W}{\partial t}+\frac{\partial W}{\partial a}- k_1(x)W_{xx} + \mu_{11}(t,x,a) W =G\chi_\omega, & (t,a,x)\in \tilde Q, 
\\[5pt]
\ds\frac{\partial Y}{\partial t}+\frac{\partial Y}{\partial a}-k_1(x)Y_{xx} + \mu_{22}(t,x,a) Y-\mu_{21}(t,x,a) W=0, & (t,a,x)\in \tilde Q,
\\
W(t, a, 0) = W(t,a, 1)  = 0, & (t,a)\in\tilde  Q_{T,A},\\
Y(t,a, 0) = Y(t, a, 1)= 0, & (t,a)\in\tilde  Q_{T,A},\\
W(T-\bar a_1, a, x) =\tilde u_0(a,x),\quad Y(T-\bar a_2,a, x) =\tilde  v_0(a,x), &  (a,x) \in Q_{A,1},\\
W(t, 0, x)=\int_0^A \beta_1 (a, x)W (t, a, x) da,  & (t,x) \in \tilde Q_{ T,1},\\
Y(t, 0, x)=\int_0^A \beta_2 (a, x)Y(t, a, x) da,  & (t,x) \in \tilde Q_{T,1},
\end{cases}
\end{equation}
Again, by Theorem \ref{esistenza}, there exists a unique solution $(W,Y)$  of \eqref{1new1} and, by the previous theorem, for every $\delta \in (\gamma,A)$
 there exist a control  $G_\delta \in L^2_{\frac{1}{k_1}}(Q)$ such that 
\[
W(T,a,x) =Y(T,a,x) =0 \quad \text{a.e. } (a,x) \in (\delta, A) \times (0,1).
\]
Now, define $(u,v)$ and $g_\delta$ by
\[u:= \begin{cases} U, & \text{in}\; [0, T_1],\\
W, & \text{in} \; [ T_1, T], \end{cases} \quad  v:=\begin{cases}V, & \text{in}\; [0, T_1], \\ Y, & \text{in}\; [T_1, T] \end{cases}
\]
and
\[g_\delta:= \begin{cases} 0, & \text{in}\; [0,  T_1],\\
G_\delta, & \text{in} \; [T_1, T]. \end{cases}
\]
Then $(u,v)$ satisfies \eqref{1} and $g_\delta  \in L^2_{\frac{1}{k_1}}(Q)$ is such that
\[
u(T,a,x) =v(T,a,x) =0 \quad \text{a.e. } (a,x) \in (\delta, A) \times (0,1).
\]
It remains to prove \eqref{stimah1}. To this aim, observe that, by  \eqref{stimah}
\begin{equation}\label{ultima1'}
\begin{aligned}
\|g_\delta\|_{L^2_{\frac{1}{k_1}}(Q)}^2 &= \int_{\tilde T_1}^T\int_0^A \int_0^1 \frac{G_\delta^2}{k_1} dxdadt  \le C \|\tilde u_0\|_{L^2_{\frac{1}{k_1}}(Q_{A,1})}^2+ \|\tilde v_0\|_{L^2_{\frac{1}{k_1}}(Q_{A,1})}^2 
\\
&=C\left( \int_0^A \int_0^1 \frac{U^2}{k_1}(T_1, a,x)dxda+  \int_0^A \int_0^1 \frac{V^2}{k_1}(T_1, a,x)dxda\right)
\end{aligned}
\end{equation}
for a  strictly positive constant $C= C_\delta$. Thus, it is sufficient to estimate the last integral in \eqref{ultima1'}. To do this,  we multiply  the two equations of \eqref{1new2} by $\ds \frac{U}{k_1}$ and $\ds \frac{V}{k_1}$, respectively; then integrating over $Q_{A,1}$, we obtain
\[
\frac{1}{2}\frac{d}{dt} \int_0^A \int_0^1 \frac{U^2}{k_1} dxda + \frac{1}{2}\int_0^1 \frac{U^2(t,A,x)}{k_1}dx + \int_0^A\int_0^1 U_x^2 dxda + \int_0^A \int_0^1 \mu_{11} \frac{U^2}{k_1} dxda=  \frac{1}{2}\int_0^1 \frac{U^2(t,0,x)}{k_1}dx
\]
and
\[
\begin{aligned}
&\frac{1}{2}\frac{d}{dt} \int_0^A \int_0^1 \frac{V^2}{k_1} dxda + \frac{1}{2}\int_0^1 \frac{V^2(t,A,x)}{k_1}dx + \int_0^A\int_0^1 V_x^2 dxda + \int_0^A \int_0^1 \mu_{22} \frac{V^2}{k_1} dxda- \int_0^A \int_0^1 \mu_{21} \frac{UV}{k_1} dxda\\
&=  \frac{1}{2}\int_0^1 \frac{V^2(t,0,x)}{k_1}dx.
\end{aligned}
\]
Hence, using the fact that $U(t,0,x)= \int_0^A \beta_1(a,x) U(t,a,x)da$ and $V(t,0,x)= \int_0^A \beta_2(a,x) V(t,a,x)da$, one has
\[
\begin{aligned}
\frac{1}{2}\frac{d}{dt} \int_0^A \int_0^1 \frac{U^2}{k_1} dxda& \le \frac{1}{2}\frac{d}{dt} \int_0^A \int_0^1 \frac{U^2}{k_1} dxda + \frac{1}{2}\int_0^1 \frac{U^2(t,A,x)}{k_1}dx + \int_0^A\int_0^1 U_x^2 dxda + \int_0^A \int_0^1 \mu_{11} \frac{U^2}{k_1} dxda
\\
&= \frac{1}{2}\int_0^1 \frac{1}{k_1}\left(\int_0^A \beta_1(a,x) U(t,a,x)da\right)^2dx\le  \frac{C}{2}\int_0^A\int_0^1 \frac{U^2}{k_1}dxda
\end{aligned}
\]
and
\[
\begin{aligned}
\frac{1}{2}\frac{d}{dt} \int_0^A \int_0^1 \frac{V^2}{k_1} dxda &\le 
\frac{1}{2}\frac{d}{dt} \int_0^A \int_0^1 \frac{V^2}{k_1} dxda + \frac{1}{2}\int_0^1 \frac{V^2(t,A,x)}{k_1}dx + \int_0^A\int_0^1 V_x^2 dxda + \int_0^A \int_0^1 \mu_{22} \frac{V^2}{k_1} dxda\\
&=  \int_0^A \int_0^1 \mu_{21} \frac{UV}{k_1} dxda + \frac{1}{2}\int_0^1 \frac{V^2(t,0,x)}{k_1}dx\\
&\le
 \frac{1}{2} \int_0^1 \frac{1}{k_1}\left(\int_0^A \beta_1(a,x) V(t,a,x)da\right)^2dx +   \frac{C}{2}\int_0^A\int_0^1 \frac{U^2}{k_1}dxda+   \frac{C}{2}\int_0^A\int_0^1 \frac{V^2}{k_1}dxda\\
& \le  \frac{C}{2}\int_0^A\int_0^1 \frac{U^2}{k_1}dxda +  C\int_0^A\int_0^1 \frac{V^2}{k_1}dxda.
\end{aligned}
\]
Thus,
\begin{equation}\label{EqU}
\frac{1}{2}\frac{d}{dt} \int_0^A \int_0^1 \frac{U^2}{k_1} dxda \le  \frac{C}{2}\int_0^A\int_0^1 \frac{U^2}{k_1}dxda
\end{equation}
and
\begin{equation}\label{EqV}
\frac{1}{2}\frac{d}{dt} \int_0^A \int_0^1 \frac{V^2}{k_1} dxda  \le  \frac{C}{2}\int_0^A\int_0^1 \frac{U^2}{k_1}dxda +  C\int_0^A\int_0^1 \frac{V^2}{k_1}dxda.
\end{equation}
Setting $F_1(t):= \|U(t)\|^2_{L^2_{\frac{1}{k_1}}(Q_{A,1})}$ and multiplying 
\eqref{EqU}
by $e^{-Ct}$, it results
\[
\frac{d}{dt} \left(e^{-Ct}F_1(t)\right) \le 0.
\]
Integrating over $(0,t)$, for all $t \in (0,T),$ we have
\begin{equation}\label{ultimo3}
\int_0^A \int_0^1 \frac{U^2(t,a,x)}{k_1} dxda \le e^{CT} \int_0^A \int_0^1 \frac{U^2(0,a,x)}{k_1} dxda=e^{CT} \int_0^A \int_0^1 \frac{u_0^2(a,x)}{k_1} dxda.
\end{equation}
In particular,
\begin{equation}\label{ultima2}
\int_0^A \int_0^1 \frac{U^2(T_1,a,x)}{k_1} dxda\le C\int_0^A \int_0^1 \frac{u_0^2(a,x)}{k_1} dxda.
\end{equation}
Now, set $F_2(t):= \|V(t)\|^2_{L^2_{\frac{1}{k_1}}(Q_{A,1})}$ and multiply \eqref{EqV} by $e^{-Ct}$, it results, by \eqref{ultimo3},
\[
\frac{d}{dt} \left(e^{-Ct}F_2(t)\right) \le Ce^{-Ct}F_1(t) \le CF_1(t) \le C\int_0^A \int_0^1 \frac{u_0^2(a,x)}{k} dxda.
\]
Integrating over $(0,t)$, for all $t \in (0,T),$ we have
\[
\begin{aligned}
\int_0^A \int_0^1 \frac{V^2(t,a,x)}{k_1} dxda &\le e^{CT} \int_0^A \int_0^1 \frac{V^2(0,a,x)}{k_1} dxda+ C\int_0^A \int_0^1 \frac{u_0^2(a,x)}{k_1} dxda\\
&\le C\int_0^A \int_0^1 \left(\frac{u_0^2(a,x)}{k_1}+\frac{v_0^2(a,x)}{k_1}\right) dxda.
\end{aligned}
\]
In particular,
\begin{equation}\label{ultima2'}
\int_0^A \int_0^1 \frac{V^2(T_2,a,x)}{k_1} dxda\le C\int_0^A \int_0^1\left(\frac{u_0^2(a,x)}{k_1}+\frac{v_0^2(a,x)}{k_1}\right) dxda.
\end{equation}
Hence, by \eqref{ultima1'},
\[
\begin{aligned}
\|g_\delta\|_{L^2_{\frac{1}{k_1}}(Q)}^2 &\le C\left( \int_0^A \int_0^1 \frac{U^2}{k_1}(T_1, a,x)dxda+  \int_0^A \int_0^1 \frac{V^2}{k_1}(T_2, a,x)dxda\right) \\
& \le  C\int_0^A \int_0^1\left(\frac{u_0^2(a,x)}{k_1}+\frac{v_0^2(a,x)}{k_1}\right) dxda.
\end{aligned}
\]
Thus \eqref{stimah1} follows.
\end{proof}

\end{document}